\documentclass[a4paper,12pt]{article}

\unitlength\textwidth
\divide\unitlength by 200\relax

\usepackage{amsmath,amssymb}
\usepackage{bm}

\bmdefine{\sss}{s}
\bmdefine{\vvv}{v}

\DeclareMathAlphabet{\mathscr}{U}{rsfs}{m}{n}

\newcommand{\msCCC}{\mathscr{C}}
\newcommand{\msOOO}{\mathscr{O}}
\newcommand{\msPPP}{\mathscr{P}}

\newcommand{\NNN}{\mathbb{N}}
\newcommand{\ZZZ}{\mathbb{Z}}
\newcommand{\QQQ}{\mathbb{Q}}
\newcommand{\RRR}{\mathbb{R}}
\newcommand{\KKK}{\mathbb{K}}

\newcommand{\RRRRR}{{\mathcal R}}
\newcommand{\SSSSS}{{\mathcal S}}

\newcommand{\covers}{\mathrel{\cdot\!\!\!>}}
\newcommand{\covered}{\mathrel{<\!\!\!\cdot}}

\newcommand{\define}{\mathrel{:=}}

\newcommand{\cm}{Cohen-Macaulay}

\newcommand{\joinirred}{join-irreducible}
\newcommand{\rank}{\mathrm{rank}}

\newcommand{\relint}{{\rm{relint}}}

\newcommand{\Div}{{\mathrm{Div}}}

\newcommand{\condn}{{condition N}}
\newcommand{\scn}{{sequence with condition N}}
\newcommand{\sscn}{{sequences with condition N}}

\newcommand{\dmax}{{d_{\max}}}

\newcommand{\xip}{\xi^{+}}
\newcommand{\xipup}{\xi^{+\uparrow}}
\newcommand{\xippup}{\xi^{+'\uparrow}}
\newcommand{\xipdown}{\xi^{+\downarrow}}
\newcommand{\xippdown}{\xi^{+'\downarrow}}
\newcommand{\etap}{\eta^{+}}
\newcommand{\etapup}{\eta^{+\uparrow}}
\newcommand{\etappup}{\eta^{+'\uparrow}}
\newcommand{\etapdown}{\eta^{+\downarrow}}
\newcommand{\etappdown}{\eta^{+'\downarrow}}

\newcommand{\dist}{{\mathrm{dist}}}

\newcommand{\qdist}[1]{{q^{(#1)}\dist}}
\newcommand{\qndist}{{q^{(n)}\dist}}

\newcommand{\qmmdist}{{q^{(-m)}\dist}}
\newcommand{\qedist}{q^{(\epsilon)}\dist}

\newcommand{\qonedist}{q^{(1)}\dist}
\newcommand{\qmonedist}{q^{(-1)}\dist}

\newcommand{\qone}{q^{(1)}}
\newcommand{\qmone}{q^{(-1)}}
\newcommand{\qe}{q^{(\epsilon)}}

\newcommand{\qonered}{$\qone$-reduced}
\newcommand{\qmonered}{$\qmone$-reduced}
\newcommand{\qered}{$q^{(\epsilon)}$-reduced}

\newcommand{\se}{\SSSSS^{(\epsilon)}}
\newcommand{\omegae}{\omega^{(\epsilon)}}

\newcommand{\kcp}{\KKK[\msCCC(P)]}
\newcommand{\kop}{\KKK[\msOOO(P)]}
\newcommand{\ekp}{E_\KKK[\msPPP]}
\newcommand{\ekcp}{E_\KKK[\msCCC(P)]}
\newcommand{\rkh}{\RRRRR_\KKK[H]}

\newtheorem{thm}{Theorem}[section]
\newtheorem{fact}[thm]{Fact}
\newtheorem{example}[thm]{Example}
\newtheorem{lemma}[thm]{Lemma}
\newtheorem{cor}[thm]{Corollary}
\newtheorem{definition}[thm]{Definition}
\newtheorem{prop}[thm]{Proposition}
\newtheorem{remark}[thm]{Remark}

\newcommand{\bigzerou}{\smash{\lower1.7ex\hbox{\bg 0}}}

\newcommand{\bigastu}{\smash{\lower1.7ex\hbox{\bg *}}}

\newcommand{\refeq}[1]{(\ref{#1})}
\numberwithin{equation}{section}

\newcommand{\mylabel}[1]{{\label{#1}\tt [#1]}}
\let\mylabel=\label

\title{%
On the canonical ideal of the Ehrhart ring of the chain polytope of a poset%
}

\author{Mitsuhiro MIYAZAKI%
}
\date{\normalsize
Department of Mathematics, Kyoto University of Education,\\
1 Fujinomori, Fukakusa, Fushimi-ku, Kyoto, 612-8522, Japan}

\begin{document}


\maketitle

\sloppy

\begin{abstract}
Let $P$ be a poset, $\msOOO(P)$ the order polytope of $P$ and 
$\msCCC(P)$ the chain polytope of $P$.
In this paper, we study the canonical ideal of the Ehrhart ring $\kcp$ of
$\msCCC(P)$ over a field $\KKK$ and characterize the level 
(resp.\ anticanonical level) property of $\kcp$ by a combinatorial structure of $P$.
In particular, we show that if $\kcp$ is level (resp.\ anticanonical level), 
then so is $\kop$.
We exhibit examples which show the converse does not hold.

Moreover, we show that the symbolic powers of the canonical ideal of $\kcp$
are identical with ordinary ones and degrees of the generators of the canonical
and anticanonical ideals are consecutive integers.
\end{abstract}

\section{Introduction}

Let $P$ be a finite partially ordered set (poset for short).
Stanley \cite{sta3} investigated two convex polytopes associated with $P$, the order polytope 
$\msOOO(P)$ and the chain polytope $\msCCC(P)$.
He showed a surprising  connection between $\msOOO(P)$ and $\msCCC(P)$.
For example 
the Ehrhart rings 
$\kop$ and $\kcp$ 
of $\msOOO(P)$ and $\msCCC(P)$
have the same Hilbert series.

About the same time, Hibi \cite{hib} constructed and studied the ring $\rkh$, where $H$ is
a finite distributive lattice, which nowadays called the Hibi ring, in the study
of algebras with straightening law (ASL).
It turned out that $\rkh$ is identical with the Ehrhart ring $\kop$ of the order
polytope $\msOOO(P)$,
where $P$ is the set of \joinirred\ elements of $H$.

After that, Hibi rings, i.e., the ring of the form $\kop$ for some poset $P$,
are extensively studied by many researchers, including the present author.
In our previous papers, we studied many aspects of Hibi rings, especially described generators
of the canonical ideals of Hibi rings and characterized level and anticanonical level properties
of Hibi rings \cite{mo,mf}.
The key notion to investigate the generators of the canonical ideal of a Hibi ring is
``\scn'' (see below).

In this paper, we investigate the canonical ideal of the Ehrhart ring $\kcp$ of the chain 
polytope $\msCCC(P)$ of a poset $P$.
The key notion in this case is ``\scn' ''.

\begin{trivlist}
\item[\hskip\labelsep\bf Definition \ref{def:condn}]
Let $y_0$, $x_1$, $y_1$, $x_2$, \ldots, $y_{t-1}$, $x_t$ be a sequence of elements 
of $P$.
We say that $y_0$, $x_1$, $y_1$, $x_2$, \ldots, $y_{t-1}$, $x_t$ satisfy \condn\ 
(resp.\ \condn') if
\begin{enumerate}
\item
$y_0>x_1<y_1>x_2<\cdots<y_{t-1}>x_t$ and
\item
if $i\leq j-2$, then $y_i\not\geq x_j$
(resp.\ $y_i\not>x_j$).
\end{enumerate}
We define that an empty sequence (i.e., $t=0$) satisfy \condn\ (resp. \condn').
\end{trivlist}
%

In the case of Hibi rings, i.e., the case of $\kop$, 
a Laurent monomial associated to a map $\nu\colon P^-\to\ZZZ$ is a generator of the
canonical (resp.\ anticanonical) ideal if and only if there is a \scn\ with an appropriate
relation with $\nu$,
where $P^-=P\cup\{-\infty\}$.
Moreover, the symbolic powers of the canonical ideal are identical with the ordinary ones
and the degrees of the generators of the canonical and anticanonical ideals are consecutive
 integers, i.e., if there are generators of the canonical (resp.\ anticanonical) ideal of
degrees $d_1$ and $d_2$ with $d_1<d_2$, then for any integer $d$ with $d_1<d<d_2$,
there is a generator of the canonical (resp.\ anticanonical) ideal with degree $d$.
Further, $\kop$ is level (resp.\ anticanonical level) if and only if 
``\qonered'' (resp.\ ``\qmonered'') \scn\ is the empty sequence only.
See Definition \ref{def:qered} for the definition of the notion \qonered\ (resp.\
\qmonered).

In this paper, we show that the almost identical phenomena occur for $\kcp$.
A Laurent monomial associated to a map $\xi\colon P^-\to\ZZZ$ is a generator
of the canonical (resp.\ anticanonical) ideal if and only if there is a \scn' with
an appropriate relation 
with $\xi$
(which is different from the one in the case of $\kop$, see Proposition \ref{prop:gen equiv}).
Moreover, the symbolic powers of the canonical ideal are identical with the ordinary
ones (Theorem \ref{thm:symbolic power})
and the generators of the canonical and anticanonical ideals are consecutive 
integers (Theorem \ref{thm:gen deg}).
Further, $\kcp$ is level (resp.\ anticanonical level) if and only if 
\qonered\ (resp.\ \qmonered) \scn' is the empty sequence only (Theorem \ref{thm:level cri}).
Since \condn' is weaker than \condn, it follows that if $\kcp$ is level (resp.\ 
anticanonical level), then so is $\kop$.
We also exhibit examples which show that the converse does not hold.
See Examples \ref{ex:level} and \ref{ex:antican level}.

\section{Preliminaries}

\mylabel{sec:pre}

In this paper, all rings and algebras are assumed to 
be commutative 
with identity element 
unless stated otherwise.

First we state notations about sets used in this paper.
We denote by $\NNN$ the set of nonnegative integers, by
$\ZZZ$ the set of integers, by
$\QQQ$ the set of rational numbers and by
$\RRR$ the set of real numbers.
We denote the cardinality of a set $X$ by $\#X$.
For nonempty sets $X$ and $Y$, we denote the set of maps from $X$ to $Y$ by $Y^X$.
If $X$ is a finite set, we identify $\RRR^X$ with the Euclidean space 
$\RRR^{\#X}$.
Let $X$ be a set  and $A$ a subset of $X$.
We define the characteristic function $\chi_A\in\RRR^X$ by
$\chi_A(x)=1$ for $x\in A$ and $\chi_A(x)=0$ for $x\in X\setminus A$.

Now we define a symbol which is frequently used in this paper.

\begin{definition}
\mylabel{def:xi+}
\rm
Let $X$ be a finite set and $\xi\in\RRR^X$.
For $B\subset X$, we set $\xip(B)\define\sum_{b\in B}\xi(b)$.
\end{definition}
Next we define operations of elements in $\RRR^X$.

\begin{definition}\rm
Let $X$ be a set.
For $\xi$, $\xi'\in\RRR^X$ and $a\in \RRR$,
we define
maps $\xi\pm\xi'$ and $a\xi$ 
by
$(\xi\pm\xi')(x)=\xi(x)\pm\xi'(x)$ and
$(a\xi)(x)=a(\xi(x))$
for $x\in X$.
\end{definition}
Note that if $X$ is a finite set, $B$ is a subset of $X$ and $a\in\RRR$, then
$(\xi\pm\xi')^+(B)=\xip(B)\pm(\xi')^+(B)$
and $(a\xi)^+(B)=a(\xip(B))$.
We denote $f\in\RRR^X$ with $f(x)=0$ for any $x\in X$ by 0.

Next we recall some definitions concerning finite partially
ordered sets (poset for short).
Let $Q$ be a finite poset.
We denote the set of maximal (resp.\ minimal) elements of $Q$
by $\max Q$ (resp.\ $\min Q$).
If $\max Q$ (resp.\ $\min Q$) consists of one element $z$, we often
abuse notation and write $z=\max Q$ (resp. $z=\min Q$).
A chain in $Q$ is a totally ordered subset of $Q$.
For a chain $X$ in $Q$, we define the length of $X$ as $\#X-1$.
The maximum length of chains in $Q$ is called the rank of $Q$ and denoted  $\rank Q$.
A subset $A$ of $Q$ is an antichain in $Q$ if every pair of two elements in $A$
are incomparable by the order of $Q$.
If $I\subset Q$ and
$x\in I$, $y\in Q$, $y\leq x\Rightarrow y\in I$,
then we say that $I$ is a poset ideal of $Q$.

Let $+\infty$ (resp.\ $-\infty$) be a new element which is not contained in $Q$.
We define a new poset $Q^+$ (resp.\ $Q^-$) whose base set is $Q\cup\{+\infty\}$
(resp.\ $Q\cup\{-\infty\}$) and
$x<y$ in $Q^+$ (resp.\ $Q^-$) if and only if $x$, $y\in Q$ and $x<y$ in $Q$
or $x\in Q$ and $y=+\infty$
(resp.\ $x=-\infty$ and $y\in Q$).
We set $Q^\pm\define(Q^+)^-$.

Let $Q'$ be an arbitrary poset. (We apply the following definition for
 $Q'=Q$, $Q^+$, $Q^-$ or $Q^\pm$.)
If $x$, $y\in Q'$, $x<y$ and there is no $z\in Q'$ with $x<z<y$,
we say that $y$ covers $x$ and denote 
$x\covered y$ or $y\covers x$.
For $x$, $y\in Q'$ with $x\leq y$, we set
$[x,y]_{Q'}\define\{z\in Q'\mid x\leq z\leq y\}$.
Further, for $x$, $y\in Q'$ with $x<y$, we set
$[x,y)_{Q'}\define\{z\in Q'\mid x\leq z< y\}$, 
$(x,y]_{Q'}\define\{z\in Q'\mid x< z\leq y\}$ and
$(x,y)_{Q'}\define\{z\in Q'\mid x< z< y\}$.
We denote $[x,y]_{Q'}$ 
(resp.\ $[x,y)_{Q'}$, $(x,y]_{Q'}$ or $(x,y)_{Q'}$) 
as $[x,y]$ 
(resp.\ $[x,y)$, $(x,y]$ or $(x,y)$)
if there is no fear of confusion.

\begin{definition}\rm
Let $Q'$ be an arbitrary finite poset and 
let $x$ and $y$ be elements of $Q'$ with $x\leq y$.
A saturated chain from $x$ to $y$ is a sequence of elements 
$z_0$, $z_1$, \ldots, $z_t$  of $Q'$ such that
$$
x=z_0\covered z_1\covered \cdots\covered z_t=y.
$$
\end{definition}
Note that the length of the chain $z_0$, $z_1$, \ldots, $z_t$ is $t$.

\begin{definition}\rm
Let $Q'$, $x$ and $y$ be as above.
We define 
$\dist(x,y)\define\min\{t\mid$ there is a saturated chain from $x$ to $y$
with length $t.\}$
and call $\dist(x,y)$ the distance of $x$ and $y$.
Further, for $n\in \ZZZ$, we define 
$\qndist(x,y)\define\max\{nt\mid$
there is a saturated chain from $x$ to $y$
with length $t.\}$
and call $\qndist(x,y)$ the $n$-th quasi-distance of $x$ and $y$.
\end{definition}
Note that $\qdist{-1}(x,y)=-\dist(x,y)$ and
$\qdist{1}(x,y)=\rank([x,y])$.
Note also that $\qndist(x,z)+\qndist(z,y)\leq\qndist(x,y)$
for any $x$, $z$, $y$ with $x\leq z\leq y$.
Further, $\qndist(x,y)=n$ if $x\covered y$,
$\qndist(x,x)=0$ and $\qdist{mn}(x,y)=m\qndist(x,y)$ for any positive integer $m$.

Now we state the following.

\begin{definition}\rm
Let $Q$ be a finite poset and $\xi\in\RRR^Q$.
For $z\in Q$, we set
$
\xipup(z)\define\max\{\sum_{\ell=0}^s\xi(z_\ell)\mid
$
there exist elements $z_0$, $z_1$, \ldots, $z_s$ with
$z=z_0\covered z_1\covered\cdots\covered z_s\covered+\infty$ in $Q^+\}$.
(resp.\
$
\xippup(z)\define\max\{\sum_{\ell=1}^s\xi(z_\ell)\mid
$
there exist elements $z_0$, $z_1$, \ldots, $z_s$ with
$z=z_0\covered z_1\covered\cdots\covered z_s\covered+\infty$ in $Q^+\}$,
where we define the empty sum to be 0).
Similarly, we set
$
\xipdown(z)\define\max\{\sum_{\ell=0}^s\xi(z_\ell)\mid
$
there exist elements $z_0$, $z_1$, \ldots, $z_s$ with
$z=z_0\covers z_1\covers\cdots\covers z_s\covers-\infty$ in $Q^-\}$
(resp.\
$
\xippdown(z)\define\max\{\sum_{\ell=1}^s\xi(z_\ell)\mid
$
there exist elements $z_0$, $z_1$, \ldots, $z_s$ with
$z=z_0\covers z_1\covers\cdots\covers z_s\covers-\infty$ in $Q^-\}$).
\end{definition}
It is clear from the definition that
$\xipup(z)=\xi(z)+\xippup(z)$,
$\xipdown(z)=\xi(z)+\xippdown(z)$ and
$\xipup(z)+\xippdown(z)=\xippup(z)+\xipdown(z)=
\max\{\xip(C)\mid C$ is a maximal chain in $Q$ with $C\ni z\}$.
Note that $\xipup_1(z)+\xipup_2(z)=(\xi_1+\xi_2)^{+\uparrow}(z)$
does not hold in general.

The following 
lemma is easily proved.

\begin{lemma}
\mylabel{lem:path}
Let $Q$ be a finite poset and $\xi\in\RRR^Q$.
\begin{enumerate}
\item
\mylabel{item:plus path}
Suppose that $z\in Q$, $z=z_0\covered z_1\covered\cdots\covered z_s\covered+\infty$ in $Q^+$
and
$\xipup(z)=\sum_{\ell=0}^s\xi(z_\ell)$
(resp.\
$z=z_0\covers z_1\covers\cdots\covers z_s\covers-\infty$ in $Q^-$
and
$\xipdown(z)=\sum_{\ell=0}^s\xi(z_\ell)$).
Then for any $k$ with $0\leq k\leq s$,
$\xipup(z_k)=\sum_{\ell=k}^s\xi(z_\ell)$ and $\xippup(z_k)=\sum_{\ell=k+1}^s\xi(z_\ell)$
(resp.\
$\xipdown(z_k)=\sum_{\ell=k}^s\xi(z_\ell)$ and $\xippdown(z_k)=\sum_{\ell=k+1}^s\xi(z_\ell)$).
\item
\mylabel{item:prime eq max}
Let $z\in Q\setminus\max Q$ (resp.\ $z\in Q\setminus \min Q$).
Then $\xippup(z)=\max\{\xipup(z')\mid z'\covers z\}$
(resp.\ $\xippdown(z)=\max\{\xipdown(z')\mid z'\covered z\}$).
\item
\mylabel{item:chain path}
Suppose that $w_1$, $w_2\in Q$, $w_1<w_2$ and 
$w_1=z_0\covered z_1\covered\cdots\covered z_s=w_2$.
Then
$\xippup(w_1)\geq\sum_{\ell=1}^{s-1}\xi(z_\ell)+\xipup(w_2)$
and
$\xippdown(w_2)\geq\sum_{\ell=1}^{s-1}\xi(z_\ell)+\xipdown(w_1)$.
\end{enumerate}
\end{lemma}
Moreover, we see the following fact.

\begin{lemma}
\mylabel{lem:plus ineq}
Let $Q$ be a finite poset,
$z\in Q$ and $n$ an integer.
Set $M\define\max\{\xip(C)\mid C$ is a maximal chain in $Q\}$.
Then 
$$
\xippup(z)+\xi(z)+\xippdown(z)
=\xipup(z)+\xippdown(z)
=\xippup(z)+\xipdown(z)
\leq M.
$$
Further, if $x$, $y\in Q$, $x<y$ and $\xi(w)\geq n$ for any $w\in (x,y)$,  then
\begin{eqnarray*}
&&\xipdown(x)+\qndist(x,y)-n+\xipup(y)\leq M,\\
&&\xippdown(y)\geq\qndist(x,y)-n+\xipdown(x)\mbox{ and}\\
&&\xippup(x)\geq\qndist(x,y)-n+\xipup(y).
\end{eqnarray*}
\end{lemma}
\begin{proof}
The first assertion is clear from the fact that
$\xipup(z)+\xippdown(z)=\max\{\xip(C')\mid C'$
is a maximal chain in $Q$ and $C'\ni z\}$.

In order to prove the other assertions, take elements
$$
z_1, \ldots, z_s, z'_0,\ldots, z'_{s'}, z''_1,\ldots, z''_{s''}
$$
of $Q$ such that
$-\infty\covered z_1\covered\cdots\covered z_s=x=z'_0\covered\cdots\covered
z'_{s'}=y=z''_1\covered\cdots\covered z''_{s''}\covered+\infty$ in $Q^\pm$,
$\xipdown(x)=\sum_{\ell=1}^s\xi(z_\ell)$,
$\qndist(x,y)=s'n$ and
$\xipup(y)=\sum_{\ell=1}^{s''}\xi(z''_\ell)$,
Then, since 
$
z_1, \ldots, z_s, z'_1,\ldots, z'_{s'-1}, z''_1,\ldots, z''_{s''}
$
is a maximal chain in $Q$ and $\xi(z'_\ell)\geq n$ for $1\leq \ell\leq s'-1$,
we see that
\begin{eqnarray*}
&&\xipdown(x)+\qndist(x,y)-n+\xipup(y)\\
&=&\xipdown(x)+(s'-1)n+\xipup(y)\\
&\leq&\sum_{\ell=1}^s\xi(z_\ell)+\sum_{\ell=1}^{s'-1}\xi(z'_\ell)
+\sum_{\ell=1}^{s''-1}\xi(z''_\ell)\\
&\leq&M.
\end{eqnarray*}
Further,
$$
\xippdown(y)\geq\sum_{\ell=1}^{s'-1}\xi(z'_\ell)+
\xipdown(x)
\geq\qndist(x,y)-n+\xipdown(x)
$$
and
$$
\xippup(x)\geq\sum_{\ell=1}^{s'-1}\xi(z'_\ell)+
\xipup(y)
\geq\qndist(x,y)-n+\xipup(y)
$$
by Lemma \ref{lem:path} \ref{item:chain path}.
\end{proof}

Next we fix notation about Ehrhart rings.
Let $X$ be a finite set with $-\infty\not\in X$ 
and $\msPPP$ a rational convex polytope in $\RRR^X$, i.e., 
a convex polytope whose vertices are contained in $\QQQ^X$.
Set $X^-\define X\cup\{-\infty\}$ and let $\{T_x\}_{x\in X^-}$ be
a family of indeterminates indexed by $X^-$.
For $f\in\ZZZ^{X^-}$, 
we define the Laurent monomial $T^f$ by
$T^f\define\prod_{x\in X^-}T_x^{f(x)}$.
We set $\deg T_x=0$ for $x\in X$ and $\deg T_{-\infty}=1$.
Then the Ehrhart ring of $\msPPP$ over a field $\KKK$ is the $\NNN$-graded subring
$$\KKK[T^f\mid f\in \ZZZ^{X^-}, f(-\infty)>0, \frac{1}{f(-\infty)}f\in\msPPP]$$
of the Laurent polynomial ring $\KKK[T_x^{\pm1}\mid x\in X][T_{-\infty}]$.
We denote the Ehrhart ring of $\msPPP$ over $\KKK$ by $\ekp$.
(We use $-\infty$ as the degree indicating element in order to be consistent with the
case of Hibi ring.)
If $\ekp$ is a standard graded algebra, i.e., generated as a $\KKK$-algebra by degree $1$
elements, we denote $\ekp$ as $\KKK[\msPPP]$.

Note that $\ekp$ is Noetherian since $\msPPP$ is rational.
Therefore normal by the criterion of Hochster \cite{hoc}.
Further, 
by the description of the canonical module of a normal affine semigroup ring
by Stanley \cite[p.\ 82]{sta2}, we see that the ideal
$$\bigoplus_{f\in\ZZZ^{X^-}, f(-\infty)>0, \frac{1}{f(-\infty)}f\in\relint\msPPP}
\KKK T^f
$$
of $\ekp$ is the canonical module of $\ekp$,
where $\relint\msPPP$ denotes the interior of $\msPPP$ in the affine span of $\msPPP$.
We call this ideal the canonical ideal of $\ekp$.

\medskip

Now we recall the definitions of order and chain polytopes \cite{sta3}.
Let $P$ be a finite poset.
The order polytope $\msOOO(P)$ and the chain polytope $\msCCC(P)$ are defined as 
follows.
$\msOOO(P)\define\{f\in\RRR^P\mid 0\leq f(x)\leq 1$ for any $x\in P$ and
if $x<y$ in $P$, then $f(x)\geq f(y)\}$,
$\msCCC(P)\define\{f\in\RRR^P\mid 0\leq f(x)$ for any $x\in P$ and 
$f^+(C)\leq 1$ for any chain in $P\}$.
The Ehrhart ring $\KKK[\msOOO(P)]$ of the order polytope of $P$ is identical with
the ring considered by Hibi \cite{hib}, which is nowadays called the Hibi ring.
We studied in our previous papers \cite{mo,mf} the canonical ideals of Hibi rings.

The main object of the study of this paper is $\ekcp$ and its canonical ideal.
We denote the canonical ideal of $\ekcp$ as $\omega$.
In order to study $\ekcp$, we first collect some basic facts on $\ekcp$.

First, the vertices of $\msCCC(P)$ are of the form $\chi_A$, where $A$ is an antichain
of $P$ (including the empty set).
Further, it is easily seen that
$\msCCC(P)=\{f\in\RRR^P\mid f(x)\geq0$ for any $x\in P$ and $f^+(C)\leq 1$ for any
maximal chain in $P\}$.
Thus, $\relint\msCCC(P)=\{f\in\RRR^P\mid f(x)>0$ for any $x\in P$ and 
$f^+(C)<1$ for any maximal chain in $P\}$.

Here, we define the following notation.

\begin{definition}
\rm
\mylabel{def:sn}
For $n\in\ZZZ$, we set
$\SSSSS^{(n)}(P)\define\{\xi\in\ZZZ^{P^-}\mid \xi(x)\geq n$
for any $x\in P$ and 
$\xi(-\infty)\geq\xip(C)+n$
for any
maximal chain $C$ in $P\}$.
\end{definition}
In the following, we fix a finite poset $P$ and abbreviate $\SSSSS^{(n)}(P)$ as $\SSSSS^{(n)}$
and apply the above definitions by setting $Q=P$, $P^+$, $P^-$ or $P^\pm$.

By the above consideration, we can describe $\ekcp$ and $\omega$ by using
this notation.
$$
\ekcp=\bigoplus_{\xi\in\SSSSS^{(0)}}\KKK T^\xi
\quad\mbox{and}\quad
\omega=\bigoplus_{\xi\in\SSSSS^{(1)}}\KKK T^\xi.
$$

Finally in this section, we note 
the following fact which seems to be a folklore.

\begin{prop}
\mylabel{prop:kcp stand}
$\ekcp$ is a standard graded $\KKK$-algebra.
\end{prop}
\begin{proof}
It is enough to show that if $\xi\in\SSSSS^{(0)}$ and $\xi(-\infty)\geq2$,
then there are $\xi_1$, $\xi_2\in\SSSSS^{(0)}$ such that $\xi_1(-\infty)=1$ and
$\xi_1+\xi_2=\xi$.

If $\xi(x)=0$ for any $x\in P$, it is enough to set $\xi_1(x)=\xi_2(x)=0$,
for any $x\in P$, $\xi_1(-\infty)=1$ and $\xi_2(-\infty)=\xi(-\infty)-1$.
Suppose that $\xi(x)>0$ for some $x\in P$.
Set $A\define\max\{x\in P\mid\xi(x)>0\}\cup\{-\infty\}$,
$\xi_1\define\chi_A$ and $\xi_2\define\xi-\xi_1$.
First it is clear from the definition that $\xi_1(-\infty)=1$ and
$\xi=\xi_1+\xi_2$.
Let $x$ be an arbitrary element of $P$.
If $x\in A$, then $\xi(x)>0$ and $\xi_1(x)=1$.
Thus, $\xi_2(x)\geq0$.
If $x\not\in A$. then $\xi_1(x)=0$ and $\xi(x)\geq0$.
Thus, $\xi_2(x)\geq0$.

Now let $C$ be an arbitrary maximal chain in $P$.
We have to show that $\xip_1(C)\leq\xi_1(-\infty)$ and
$\xip_2(C)\leq\xi_2(-\infty)$.
If $C\cap A\neq\emptyset$, then there is a unique element $x\in C\cap A$.
Therefore, $\xip_1(C)=1=\xi_1(-\infty)$ and
$\xip_2(C)=\xip(C)-1\leq\xi(-\infty)-1=\xi_2(-\infty)$.

Suppose that $C\cap A=\emptyset$.
If $\xi(c)=0$ for any $c\in C$, $\xip_1(C)=\xip(C)=0$.
Thus, $\xip_1(C)\leq 1=\xi_1(-\infty)$ and
$\xip_2(C)=\xip(C)=0\leq\xi_2(-\infty)$.
Assume that $\xi(c)>0$ for some $c\in C$.
Set $x\define\max\{c\in C\mid\xi(c)>0\}$.
Since $x\not\in A$, there exists $y\in A$ with $y>x$.
Take a saturated chain $x=z_0\covered\cdots\covered z_s=y$
and $y=z'_0\covered\cdots\covered z'_{s'}\covered+\infty$.
Then
$C'\define(C\cap(-\infty,x])\cup\{z_1,\ldots, z_s, z'_1, \ldots, z'_{s'}\}$
is a maximal chain in $P$.
Further, $\xip(C')>\xip(C)$ since $\xi(c)=0$ for any $c\in C$ with $c>x$ and
$C'\ni z_s=y\in A$.
Therefore, $\xip(C)\leq\xip(C')-1\leq\xi(-\infty)-1$.
Since $C\cap A=\emptyset$, $\xip_1(C)=\chi^+_A(C)=0\leq 1=\xi_1(-\infty)$
and
$\xip_2(C)=\xip(C)-\xip_1(C)=\xip(C)\leq\xi(-\infty)-1=\xi_2(-\infty)$.

Thus, we see that $\xi_1$, $\xi_2\in\SSSSS^{(0)}$.
\end{proof}

\section{Symbolic powers of the canonical ideal}

\mylabel{sec:symbolic power}

In this section, we consider the symbolic powers $\omega^{(n)}$ of the canonical ideal
$\omega$ of $\kcp$ and describe the Laurent monomial basis of $\omega^{(n)}$
as a vector space over $\KKK$.
Further, we show that the symbolic powers of $\omega$ are equal to the ordinary ones.

Here we note the following fact.
Let $R$ be a Noetherian normal domain and $I$ a fractional ideal.
$I$ is said to be divisorial if $R:_{Q(R)}(R:_{Q(R)}I)=I$, i.e.,
$I$ is reflexive as an $R$-module,
where $Q(R)$ is the fraction field of $R$.
It is known that the set of divisorial ideals form a group,
denoted $\Div(R)$, by the operation
$I\cdot J\define R:_{Q(R)}(R:_{Q(R)}IJ)$.
We denote the $n$-th power of $I$ in this group $I^{(n)}$, 
where $n\in\ZZZ$.
Note that if $I\subsetneq R$, then $I^{(n)}$ is identical with the 
$n$-th symbolic power of $I$.
Note also that the inverse element of $I$ in $\Div(R)$ is $R:_{Q(R)}I$.

Suppose further that $R$ is an affine semigroup ring
generated by Laurent monomials 
in the Laurent polynomial ring $\KKK[X_1^{\pm1}, \ldots, X_s^{\pm1}]$
over $\KKK$,
where $\KKK$ is a field and $X_1$, \ldots, $X_s$ are indeterminates.
Let $I$ be a divisorial ideal generated by 
Laurent monomials
$m_1$, \ldots, $m_\ell$.
Then $R:_{Q(R)}I=\bigcap_{i=1}^\ell Rm_i^{-1}$.
Thus, $R:_{Q(R)}I$ is an $R$-submodule of 
of $\KKK[X_1^{\pm1},\ldots, X_s^{\pm1}]$
generated by Laurent monomials.
Therefore, the set of divisorial ideals generated by 
Laurent monomials
 form a subgroup of $\Div(R)$.
It is known that the canonical module $\omega$ is reflexive and isomorphic to an ideal.
Therefore
$\omega\in\Div(R)$.
Thus, if the  canonical module $\omega$ of
$R$ is isomorphic to an ideal of $R$ generated by 
Laurent monomials, then
the inverse element $\omega^{(-1)}$ of $\omega$ in $\Div(R)$ 
is also an $R$-submodule
$\KKK[X_1^{\pm1},\ldots, X_s^{\pm1}]$
generated by 
Laurent monomials.

Since an $R$-submodule of $\KKK[X_1^{\pm1},\ldots, X_s^{\pm1}]$ 
generated by Laurent monomials
has a unique system
of generators consisting of Laurent monomials,
we call an element of this system a generator of the $R$-submodule.

Taking into account of this fact,
we first note the following fact.

\begin{lemma}
\mylabel{lem:an sn elem}
Let $n$ be an integer.
Set $\xi\in\ZZZ^{P^-}$ by
$$
\xi(z)\define
\begin{cases}
n,&\mbox{if $z\in P$;}\\
\qndist(-\infty,\infty),&\mbox{if $z=-\infty$,}
\end{cases}
$$
then $\xi\in\SSSSS^{(n)}$.
\end{lemma}
\begin{proof}
It is clear that $\xi(z)\geq n$ for any $z\in P$.
Let $C$ be an arbitrary maximal chain in $P$.
Then $C\cup\{-\infty,\infty\}$ is a saturated chain from $-\infty$ to $\infty$
of length 
$\#(C\cup\{-\infty,\infty\})-1$.
Thus, $n(\#(C\cup\{-\infty,\infty\})-1)\leq\qndist(-\infty,\infty)=\xi(-\infty)$.
Therefore, $\xip(C)+n=n(\#C)+n\leq\xi(-\infty)$.
\end{proof}
Next we state the following.

\begin{prop}
\mylabel{prop:in omega n}
Let $n$ be an integer and $\xi\in\ZZZ^{P^-}$.
Then $T^\xi\in\omega^{(n)}$ if and only if $\xi\in\SSSSS^{(n)}$.
In particular,
$$
\omega^{(n)}=\bigoplus_{\xi\in\SSSSS^{(n)}}\KKK T^\xi.
$$
\end{prop}
\begin{proof}
The cases where $n=0$ and $1$ are noted in the previous section.
Let $m$ be a positive integer.

We first show that $T^\xi\in\omega^{(-m)}$ if and only if $\xi\in\SSSSS^{(-m)}$.
In order to prove the ``only if'' part, first let $\eta$ be a map from $P^-$ to
$\ZZZ$ such that $\eta(z)=1$ for any $z\in P$ and $\eta(-\infty)=\qonedist(-\infty,\infty)$.
Then $T^\eta\in\omega$ by Lemma \ref{lem:an sn elem}.
Thus, $T^{m\eta}\in\omega^m$.
Since $T^\xi\in\omega^{(-m)}$, we see that $T^{\xi+m\eta}\in\kcp$.
Therefore, $\xi(z)+m\eta(z)\geq0$ for any $z\in P$ and we see that
$\xi(z)\geq-m$ since $\eta(z)=1$.
Next, let $C$ be an arbitrary maximal chain in $P$,
$N$ a huge integer ($N>\qonedist(-\infty,\infty)$) and
let $\eta'$ be a map from $P^-$ to $\ZZZ$ such that $\eta'(z)=N$ for $z\in C$,
$\eta'(z)=1$ for $z\in P\setminus C$ and $\eta'(-\infty)=N(\#C)+1$.
Then $T^{\eta'}\in\omega$ since $N$ is a huge integer and therefore
$T^{m\eta'}\in\omega^m$.
Thus, $T^{\xi+m\eta'}\in\kcp$ since $T^\xi\in\omega^{(-m)}$.
Therefore,
$\xi(-\infty)+m\eta'(-\infty)\geq\xip(C)+m((\eta')^+(C))$.
Since $\eta'(-\infty)=N(\#C)+1$ and $(\eta')^+(C)=N(\#C)$,
we see that
$\xi(-\infty)-\xip(C)\geq m((\eta')^+(C))-m\eta'(-\infty)=-m$,
i.e.,
$\xi(-\infty)\geq\xip(C)-m$.

Next, we prove the ``if'' part.
Let $T^{\eta_1}$, \ldots, $T^{\eta_m}$ be arbitrary Laurent monomials in $\omega$,
where $\eta_i\in\ZZZ^{P^-}$ for $1\leq i\leq m$.
Since $\eta_i(z)\geq 1$ and $\xi(z)\geq -m$
for any $z\in P$ and $i$, we see that
$(\xi+\eta_1+\cdots+\eta_m)(z)\geq 0$ for any $z\in P$.
Next let $C$ be an arbitrary maximal chain in $P$.
Since $\eta_i^+(C)+1\leq\eta_i(-\infty)$ for any $i$,
we see that
$(\eta_1+\cdots+\eta_m)^+(C)+m\leq(\eta_1+\cdots+\eta_m)(-\infty)$.
Since $\xi(-\infty)\geq\xip(C)-m$ by assumption, we see that
$(\xi+\eta_1+\cdots+\eta_m)(-\infty)
\geq
(\xi+\eta_1+\cdots+\eta_m)^+(C)$.
Therefore, we see that $T^{\xi+\eta_1+\cdots+\eta_m}\in\kcp$.
Since $T^{\eta_1}$, \ldots, $T^{\eta_m}$ are arbitrary Laurent monomials
in $\omega$, we see that $T^\xi\in\omega^{(-m)}$.

We next show that for $\xi\in\ZZZ^{P^-}$,
$T^\xi\in\omega^{(m)}$ if and only if $\xi\in\SSSSS^{(m)}$.
We first show the ``only if'' part.
Set $\zeta\in\ZZZ^{P^-}$ by $\zeta(z)=-m$ for any $z\in P$ and
$\zeta(-\infty)=\qmmdist(-\infty,\infty)$.
Then $\zeta\in\SSSSS^{(-m)}$ by Lemma \ref{lem:an sn elem} and therefore
we see that $T^\zeta\in\omega^{(-m)}$
by the fact proved above.
Since $T^{\xi+\zeta}\in\kcp$ by assumption, we see that
$\xi(z)+\zeta(z)\geq0$ for any $z\in P$.
Thus, $\xi(z)\geq-\zeta(z)=m$ for any $z\in P$.
Next let $C$ be an arbitrary maximal chain in $P$.
Set $\zeta'\in\ZZZ^{P^-}$ by
$$
\zeta'(z)=
\begin{cases}
0,&\mbox{if $z\in C$;}\\
-m,&\mbox{otherwise.}
\end{cases}
$$
Then 
$\zeta'\in\SSSSS^{(-m)}$ and therefore
by the criterion of a Laurent monomial to be an element of $\omega^{(-m)}$ proved above, 
we see that 
$T^{\zeta'}\in\omega^{(-m)}$.
Thus, $T^{\xi+\zeta'}\in\kcp$ by assumption
and  we see that
$\xip(C)+(\zeta')^+(C)\leq\xi(-\infty)+\zeta'(-\infty)$.
Therefore,
$\xi(-\infty)\geq\xip(C)+(\zeta')^+(C)-\zeta'(-\infty)=\xip(C)+m$.

Next we show the ``if'' part.
Let $\xi$ be a map from $P^-$ to $\ZZZ$ such that $\xi(z)\geq m$ for any $z\in P$
and $\xi(-\infty)\geq\xip(C)+m$ for any maximal chain $C$ in $P$.
Let $T^\zeta$, $\zeta\in\ZZZ^{P^-}$ be an arbitrary Laurent monomial in
$\omega^{(-m)}$.
Then by the first part of this proof, we see that $\zeta(z)\geq -m$ for any
$z\in P$ and $\zeta(-\infty)\geq\zeta^+(C)-m$ for any maximal chain $C$ in $P$.
Therefore, $(\xi+\zeta)(z)=\xi(z)+\zeta(z)\geq0$ and
$(\xi+\zeta)(-\infty)=\xi(-\infty)+\zeta(-\infty)\geq
\xip(C)+m+\zeta^+(C)-m=(\xi+\zeta)^+(C)$
for any maximal chain $C$ in $P$.
Thus, we see that $T^{\xi+\zeta}\in\kcp$.
Since $T^\zeta$ is an arbitrary Laurent monomial in $\omega^{(-m)}$,
we see that $T^\xi\in\omega^{(m)}$.
\end{proof}
As a corollary, we see the following fact.

\begin{cor}
\mylabel{cor:gen deg min}
Let $n$ be an integer and let $\xi\colon P^-\to\ZZZ$ be
the map defined in Lemma \ref{lem:an sn elem}.
Then $T^\xi$ is an element of $\omega^{(n)}$ with $\deg T^\xi=\qndist(-\infty,\infty)$
and for any map $\zeta\colon P^-\to\ZZZ$
with $T^\zeta\in\omega^{(n)}$,
$\deg T^\zeta\geq\deg T^\xi$.
\end{cor}
\begin{proof}
By Lemma \ref{lem:an sn elem},
%
we see that $\xi\in\SSSSS^{(n)}$.
Thus, 
$T^\xi\in\omega^{(n)}$ by Proposition \ref{prop:in omega n}
and $\deg T^\xi=\xi(-\infty)=\qndist(-\infty,+\infty)$.
Also, by Proposition \ref{prop:in omega n}, we see that $\zeta\in\SSSSS^{(n)}$.
Take a maximal chain $C$ in $P$ with
$\qndist(-\infty,\infty)=n(\#(C\cup\{-\infty,\infty\})-1)$.
Then $\qndist(-\infty,\infty)=n(\#C)+n$.
Since $\zeta(c)\geq n$ for any $c\in C$, we see that
$\zeta^+(C)\geq n(\#C)$.
Thus, $\deg T^\zeta=\zeta(-\infty)\geq\zeta^+(C)+n=\qndist(-\infty,\infty)=\deg T^\xi$.
\end{proof}

Next, we show that the symbolic powers of $\omega$ are equal to ordinary ones.
For $\alpha\in\RRR$, we denote by $\lfloor\alpha\rfloor$ the maximum integer 
less than or equals to
$\alpha$ and by $\lceil\alpha\rceil$ the minimum integer 
larger than or equals to $\alpha$.
Let $\epsilon$ be $1$ or $-1$,
$n$ an integer with $n\geq 2$ and $T^\xi$ an arbitrary Laurent monomial in 
$\omega^{(n\epsilon )}$, where $\xi\in\ZZZ^{P^-}$.
In this setting, we define two maps $\xi_1$ and $\xi_2\in\ZZZ^{P^-}$ as follows.
Set $\xi_1$ by
$$
\xi_1(w)=
\begin{cases}
\lfloor\frac{1}{n}\xi(w)\rfloor,&\vtop{\hsize=.4\textwidth\relax\noindent
if $w$ is a maximal element of $P$ or $w=-\infty$;}\\
\lfloor\frac{1}{n}\xipup(w)\rfloor-\lfloor\frac{1}{n}\xippup(w)\rfloor,&\mbox{otherwise}
\end{cases}
$$
and $\xi_2=\xi-\xi_1$.
We first show the following.

\begin{lemma}
\mylabel{lem:xi1 plus}
In the above notation,
$\xipup_1(z)=\lfloor\frac{1}{n}\xipup(z)\rfloor$
for any $z\in P$
and 
$\max\{\xip_1(C)\mid C$ is a maximal chain in $P\}=
\lfloor\frac{1}{n}\max\{\xip(C)\mid C$ is a maximal chain in $P\}\rfloor$.
\end{lemma}
\begin{proof}
First we show the former assertion by Noetherian induction on $z$.
The case where $z$ is a maximal element of $P$ is trivial.
Assume that $z$ is not maximal.
Then by Lemma \ref{lem:path} \ref{item:prime eq max}, we see that
$\xippup(z)=\max\{\xipup(z')\mid z'\covers z\}$.
Since
$\lfloor\frac{1}{n}\max\{\xipup(z')\mid z'\covers z\}\rfloor=
\max\{\lfloor\frac{1}{n}\xipup(z')\rfloor\mid z'\covers z\}$,
we see by the induction hypothesis and Lemma \ref{lem:path} \ref{item:prime eq max}
that
$\lfloor\frac{1}{n}\xippup(z)\rfloor=\max\{\xipup_1(z')\mid z'\covers z\}
=\xippup_1(z)$.
Therefore, by the definition of $\xi_1$, we see that
$\xipup_1(z)=\xi_1(z)+\xippup_1(z)
=\lfloor\frac{1}{n}\xipup(z)\rfloor-\lfloor\frac{1}{n}\xippup(z)\rfloor
+\lfloor\frac{1}{n}\xippup(z)\rfloor
=\lfloor\frac{1}{n}\xipup(z)\rfloor$.

Next we show the latter assertion.
Since $\max\{\xip(C)\mid C$ is a maximal chain in $P\}
=\max\{\xipup(z)\mid z$ is a minimal element of $P\}$
and the corresponding equation for $\xi_1$, we see that
$\max\{\xip_1(C)\mid C$ is a maximal chain in $P\}
=\max\{\xipup_1(z)\mid z$ is a minimal element of $P\}
=\max\{\lfloor\frac{1}{n}\xipup(z)\rfloor\mid z$ is a minimal element of $P\}
=\lfloor\frac{1}{n}\max\{\xip(C)\mid C$ is a maximal chain in  $P\}\rfloor
$.
\end{proof}
Now we prove the following.

\begin{lemma}
\mylabel{lem:xi1 in s1}
In the above setting, $\xi_1\in\SSSSS^{(\epsilon)}$.
\end{lemma}
\begin{proof}
It is enough to show that $\xi_1(z)\geq \epsilon$ for any $z\in P$ and
$\xi_1(-\infty)\geq\max\{\xip_1(C)\mid C$ is a maximal chain in $P\}+\epsilon$.

Let $z$ be an arbitrary element of $P$.
If $z$ is a maximal element of $P$, then
$\xi_1(z)=\lfloor\frac{1}{n}\xi(z)\rfloor\geq\epsilon$
since $\xi(z)\geq n\epsilon$
by Proposition \ref{prop:in omega n}.
Assume that $z$ is not a maximal element of $P$.
Then $\xipup(z)=\xi(z)+\xippup(z)\geq n\epsilon +\xippup(z)$.
Therefore, 
$\xi_1(z)=\lfloor\frac{1}{n}\xipup(z)\rfloor-\lfloor\frac{1}{n}\xippup(z)\rfloor
\geq\lfloor\frac{1}{n}(n\epsilon +\xippup(z))\rfloor-\lfloor\frac{1}{n}\xippup(z)\rfloor
=\epsilon$.

Further, since $T^\xi\in\omega^{(n\epsilon )}$, we see by Proposition
\ref{prop:in omega n} that
$\xi(-\infty)\geq\max\{\xip(C)\mid C$ is a maximal chain in $P\}+n\epsilon $.
Thus,
$\xi_1(-\infty)=\lfloor\frac{1}{n}\xi(-\infty)\rfloor
\geq\lfloor\frac{1}{n}(\max\{\xip(C)\mid C$ is a maximal chain in $P\}+n\epsilon)\rfloor
=\max\{\xip_1(C)\mid C$ is a maximal chain in $P\}+\epsilon$
by Lemma \ref{lem:xi1 plus}.
\end{proof}

Next we show that $\xi_2\in\SSSSS^{((n-1)\epsilon)}$.
First we state the following.

\begin{lemma}
\mylabel{lem:xi2 plus}
In the above setting, 
$\xipup_2(z)=\lceil\frac{n-1}{n}\xipup(z)\rceil$ for any $z\in P$ and
$\max\{\xip_2(C)\mid C$ is a maximal chain in $P\}
=\lceil\frac{n-1}{n}\max\{\xip(C)\mid C$ is a maximal chain in $P\}\rceil$.
\end{lemma}
\begin{proof}
We first prove the former assertion by Noetherian induction on $z$.
If $z$ is a maximal element of $P$, then
$\xipup_2(z)=\xi_2(z)=\xi(z)-\xi_1(z)=\xi(z)-\lfloor\frac{1}{n}\xi(z)\rfloor
=\lceil\frac{n-1}{n}\xi(z)\rceil=\lceil\frac{n-1}{n}\xipup(z)\rceil$
by the definition of $\xi_1$ and $\xi_2$.
Suppose that $z$ is not a maximal element of $P$.
Take $z_1\in P$ with $z_1\covers z$ and $\xipup(z_1)=\max_{z'\covers z}\xipup(z')$.
Then since $\xipup_2(z)=\xi_2(z)+\max_{z'\covers z}\xipup_2(z')$,
we see by induction hypothesis that
$\xipup_2(z)=\xi_2(z)+\max_{z'\covers z}\lceil\frac{n-1}{n}\xipup(z')\rceil
=\xi_2(z)+\lceil\frac{n-1}{n}\xipup(z_1)\rceil$.
Since
$\xi_2(z)=\xi(z)-\xi_1(z)
=\xi(z)-(\lfloor\frac{1}{n}\xipup(z)\rfloor-\lfloor\frac{1}{n}\xippup(z)\rfloor)
=\xi(z)-\lfloor\frac{1}{n}\xipup(z)\rfloor+\lfloor\frac{1}{n}\xipup(z_1)\rfloor$,
we see that
$\xipup_2(z)
=\xi_2(z)+\lceil\frac{n-1}{n}\xipup(z_1)\rceil
=\xi(z)-\lfloor\frac{1}{n}\xipup(z)\rfloor+\lfloor\frac{1}{n}\xipup(z_1)\rfloor
+\lceil\frac{n-1}{n}\xipup(z_1)\rceil
=\xi(z)+\xipup(z_1)-\lfloor\frac{1}{n}\xipup(z)\rfloor
=\xipup(z)-\lfloor\frac{1}{n}\xipup(z)\rfloor
=\lceil\frac{n-1}{n}\xipup(z)\rceil$.

Now we prove the latter assertion.
By the fact
$\max\{\xip(C)\mid C$ is a maximal chain in $P\}
=\max\{\xipup(z)\mid z$ is a minimal element of $P\}$
and the corresponding fact for $\xi_2$,
we see that
$\max\{\xip_2(C)\mid C$ is a maximal chain in $P\}
=\max\{\xipup_2(z)\mid z$ is a minimal element of $P\}
=\max\{\lceil\frac{n-1}{n}\xipup(z)\rceil\mid z$ is a minimal element of $P\}
=\lceil\frac{n-1}{n}\max\{\xip(C)\mid C$ is a maximal chain in $P\}\rceil$.
\end{proof}
Now we show the following.

\begin{lemma}
\mylabel{lem:xi2 in s n-1}
$\xi_2\in\SSSSS^{((n-1)\epsilon)}$.
\end{lemma}
\begin{proof}
Let $z\in P$.
If $z$ is a maximal element of $P$, then
by Lemma \ref{lem:xi2 plus}
$\xi_2(z)=\xipup_2(z)
=\lceil\frac{n-1}{n}\xipup(z)\rceil
=\lceil\frac{n-1}{n}\xi(z)\rceil
\geq(n-1)\epsilon$,
since $\xi(z)\geq n\epsilon$.
Next suppose that $z$ is not a maximal element of $P$.
Take $z_1\in P$ with $z_1\covers z$ and $\xipup(z_1)=\max_{z'\covers z}\xipup(z')
=\xippup(z)$.
Then by Lemma \ref{lem:xi2 plus}, we see that
$\xipup_2(z_1)=\max_{z'\covers z}\xipup_2(z')=\xippup_2(z)$.
Therefore, by Lemma \ref{lem:xi2 plus}, we see that
$\xi_2(z)=\xipup_2(z)-\xipup_2(z_1)
=\lceil\frac{n-1}{n}\xipup(z)\rceil-\lceil\frac{n-1}{n}\xipup(z_1)\rceil
\geq (n-1)\epsilon$,
since
$\xipup(z)=\xi(z)+\xipup(z_1)
\geq n\epsilon+\xipup(z_1)$.

Further, since 
$\xi(-\infty)\geq\max\{\xip(C)\mid C$ is a maximal chain in $P\}+n\epsilon$,
we see by Lemma \ref{lem:xi2 plus} that
$\xi_2(-\infty)=\xi(-\infty)-\xi_1(-\infty)
=\xi(-\infty)-\lfloor\frac{1}{n}\xi(-\infty)\rfloor
=\lceil\frac{n-1}{n}\xi(-\infty)\rceil
\geq\lceil\frac{n-1}{n}(\max\{\xip(C)\mid C$ is a maximal chain in $P\}+n\epsilon)\rceil
=\max\{\xip_2(C)\mid C$ is a maximal chain in $P\}+(n-1)\epsilon$.
Thus, we see that $\xi_2\in\SSSSS^{((n-1)\epsilon)}$.
\end{proof}

By Proposition \ref{prop:in omega n}, Lemmas \ref{lem:xi1 in s1}  and \ref{lem:xi2 in s n-1},
and induction on $n$, we see the following.

\begin{thm}
\mylabel{thm:symbolic power}
Let $n$ be a positive integer.
Then
$$
\omega^{(n)}=
\omega^{n}=\bigoplus_{\xi\in\SSSSS^{(n)}}\KKK T^\xi
$$
and
$$
\omega^{(-n)}=(\omega^{(-1)})^n=
\bigoplus_{\xi\in\SSSSS^{(-n)}}\KKK T^\xi.
$$
\end{thm}

\begin{remark}
\rm
Let $P=\{x,y\}$ be a poset with $x<y$ and let $\xi\colon P^-\to\ZZZ$
be a map with $\xi(x)=\xi(y)=3$ and $\xi(-\infty)=8$.
Then $\xi\in\SSSSS^{(2)}$.
If we set $\xi'(w)=\lfloor\frac{1}{2}\xi(w)\rfloor$ for any $w\in P^-$,
then $\xi'\in\SSSSS^{(1)}$ but $\xi-\xi'\not\in\SSSSS^{(1)}$.
Thus, simply setting $\xi_1(w)=\lfloor\frac{1}{n}\xi(w)\rfloor$ for any
$w\in P^-$ does not meet our demand.

Further, let $P_1=\{x_1, x_2, y\}$ be a poset with order relation $x_1<x_2$
and let $\xi\colon P^-\to \ZZZ$ be the map with
$\xi(x_1)=\xi(x_2)=2$, $\xi(y)=4$ and $\xi(-\infty)=6$.
Then $\xi\in\SSSSS^{(2)}$.
If we set $\xi'(z)=1$ for $z\in P$ and $\xi'(-\infty)=3$, then $\xi'\in \SSSSS^{(1)}$
but $\xi-\xi'\not\in\SSSSS^{(1)}$.
Thus, setting $\xi_1(z)=1$ for any $z\in P$ and $\xi_1(-\infty)=\qedist(-\infty,\infty)$
does not meet our demand also.
\end{remark}

For a poset $Q$, we define a graph $G(Q)$ whose vertex set is $Q$ and for
$a$, $b\in Q$, $\{a, b\}$ is an edge of $G(Q)$ if and only if $a\neq b$ and
$a$ and $b$ are comparable by the order of $Q$.
A graph $G$ such that there is a poset $Q$ with $G=G(Q)$ is called a 
comparability graph.
Recall that for a finite graph $G$ with vertex set $V$ and edge set $E$,
a stable set of $G$ is a subset of $V$ pairwise nonadjacent.
Further, the stable set polytope (vertex-packing polytope) of $G$ is the convex
polytope in $\RRR^V$ which is the convex hull of $\{\chi_A\mid A$ is a stable set$\}$.

Since in our setting, $\msCCC(P)$ is the convex hull of 
$\{\chi_A\mid A$ is an antichain of $P\}$ and a subset of $P$ is an antichain if and
only if it is a stable set of $G(P)$, we see that $\msCCC(P)$ is the stable set polytope
of $G(P)$.
Further, it is known the following fact.

\begin{fact}[{\cite[Theorem 1]{gh}}]
Let $G$ be a graph.
Then $G$ is a comparability graph of a poset if and only if each odd cycle of
$G$ has at least one triangular chord.
\end{fact}

By the consideration above, we see the following.

\begin{cor}
\mylabel{cor:comp graph}
Let $G$ be a finite graph.
Suppose that each odd cycle of $G$ has a triangular chord.
Denote the canonical ideal of the Ehrhart ring of the stable set polytope of
$G$ as $\omega$.
Then for any positive integer $n$,
$\omega^{(n)}=\omega^n$ and $\omega^{(-n)}=(\omega^{(-1)})^n$.
In particular, the symbolic power of the canonical ideal of
the Ehrhart ring of the stable set polytope
of a bipartite graph is identical with the ordinary power.
\end{cor}

%
%
\section{Characterizations of level and anticanonical level properties of $\kcp$}

In this section, we describe generators of $\omega$ (resp.\ $\omega^{(-1)}$)
by ``\scn' " in $P$.
Recall that we \cite{mo,mf} described generators of the canonical and anticanonical
ideals of $\kop$ by ``\scn" in $P$.

First we recall the definition of \condn\ and define \condn' of a sequence
of elements of $P$.

\begin{definition}
\rm
\mylabel{def:condn}
Let $y_0$, $x_1$, $y_1$, $x_2$, \ldots, $y_{t-1}$, $x_t$ be a sequence of elements 
of $P$.
We say that $y_0$, $x_1$, $y_1$, $x_2$, \ldots, $y_{t-1}$, $x_t$ satisfy \condn\ 
(resp.\ \condn') if
\begin{enumerate}
\item
$y_0>x_1<y_1>x_2<\cdots<y_{t-1}>x_t$ and
\item
if $i\leq j-2$, then $y_i\not\geq x_j$
(resp.\ $y_i\not>x_j$).
\end{enumerate}
We define that an empty sequence (i.e., $t=0$) satisfy \condn\ (resp. \condn').
\end{definition}
When considering a sequence $y_0$, $x_1$, \ldots, $y_{t-1}$, $x_t$ with \condn\
(resp.\ \condn'), we usually set $x_0=-\infty$ and $y_t=\infty$ and consider a
sequence $x_0$, $y_0$, $x_1$, \ldots, $y_{t-1}$, $x_t$, $y_t$ in $P^\pm$.

We note the following basic fact because it is used frequently.

\begin{lemma}
\mylabel{lem:cond np basic}
Let $y_0$, $x_1$, \ldots, $y_{t-1}$, $x_t$ be a \scn'.
If $y_i\leq y_j$ (or $x_i\leq x_j$) and $i\neq j$, then $i<j$.
In particular, $y_0$, $y_1$, \ldots, $y_{t-1}$ (resp.\ $x_1$, $x_2$, \ldots, $x_t$)
are different elements of $P$.
Further, if $y_k\leq x_\ell$, then $\ell\geq k+2$.
\end{lemma}
\begin{proof}
If $j<i$, then 
$y_j\geq y_i>x_{i+1}$
(or $y_{j-1}>x_j\geq x_i$)
and $j\leq (i+1)-2$
(or $j-1\leq i-2$),
contradicting \condn'.
Similarly, if $\ell\leq k$, then $\ell-1\leq(k+1)-2$ and 
$y_{\ell-1}>x_\ell\geq y_k>x_{k+1}$ violating \condn'.
Further, since $x_{k+1}<y_k$, we see $\ell\neq k+1$.
Thus, $\ell\geq k+2$.
\end{proof}

Here we recall the definition of level (resp.\ anticanonical level) property.

\begin{definition}[\cite{sta1,pag}]\rm
\mylabel{def:anticanonical level}
Let $R$ be a standard graded \cm\ algebra over a field.
If the degree of all the generators of the canonical module $\omega$ of $R$ are the same,
then we say that $R$ is level.
Moreover, if $R$ is normal (thus, is a domain) 
and the degree of all the generators of $\omega^{(-1)}$ 
are the same, we say that $R$ is anticanonical level.
\end{definition}

In the following, we characterize if a Laurent monomial in $\omega$
(resp.\ $\omega^{(-1)}$) is a generator by \sscn' in $P$ and also characterize
level and anticanonical level properties of $\kcp$ by \sscn' in $P$.
Since the cases of $\omega$ and $\omega^{(-1)}$ are similar, we treat both
cases simultaneously.

In the following, let $\epsilon$ be $1$ or $-1$.
Further, we define the following notation.

\begin{definition}
\rm
Let $\xi\in \ZZZ^P$, 
$\xi\in \ZZZ^{P^-}$, 
$\xi\in \ZZZ^{P^+}$ or
$\xi\in \ZZZ^{P^\pm}$ and $n\in \ZZZ$.
We set
$C_\xi^{[n]}\define\{ C\mid C$ is a maximal chain in $P$ and 
$\xip(C)=n\}$.
\end{definition}
Note that if $\xi\in\SSSSS^{(m)}$ and $\xi(-\infty)=d$, then $C_\xi^{[n]}=\emptyset$
 for any $n$ with $n>d-m$.

Next we note the following.

\begin{lemma}
\mylabel{lem:not a gen}
Suppose that $\xi\in\se$ and set $d\define\xi(-\infty)$.
Then $T^\xi$ is not a generator of $\omegae$ if and only if there
is an antichain $A$ (which may be an empty set) such that
for any $C\in C_\xi^{[d-\epsilon]}$,
$C\cap A\neq\emptyset$ and $\xi(a)>\epsilon$ for the unique element $a\in C\cap A$.
(This condition is trivially satisfied if $C_\xi^{[d-\epsilon]}=\emptyset$.)
\end{lemma}
\begin{proof}
We first prove the ``if'' part.
Set $\xi_1\colon P^-\to\ZZZ$ by
$$
\xi_1(z)\define
\begin{cases}
1,&\mbox{if $z\in A\cap(\bigcup_{C\in C_\xi^{[d-\epsilon]}}C)$ or $z=-\infty$;}\\
0,&\mbox{otherwise}
\end{cases}
$$
and $\xi_2\define\xi-\xi_1$.
Then $\xi_2(z)\geq \epsilon$ for any $z\in P$ by assumption.
Further, $C_{\xi_2}^{[m]}=\emptyset$ for any $m$ with 
$m\geq d-\epsilon$ since $A\cap C\neq\emptyset$
for any $C\in C_\xi^{[d-\epsilon]}$ and $C_\xi^{[n]}=\emptyset$ for any $n$ with
$n>d-\epsilon$.
Thus, $\xip_2(C)\leq d-\epsilon -1$ for any maximal chain $C$ in $P$.
Therefore, $\xi_2\in\se$ since $\xi_2(-\infty)=d-1$
and we see that $T^{\xi_2}\in\omegae$.
Since $\xi_1\in\SSSSS^{(0)}$, we see that 
$T^{\xi_1}\in\kcp$ and therefore
$T^\xi$ is not a generator of $\omegae$,
since $T^\xi=T^{\xi_1}T^{\xi_2}$.

Now we prove the ``only if'' part.
Suppose that $T^\xi$ is not a generator of $\omegae$.
Then, since $\kcp$ is a standard graded algebra, there are
$\xi_1\in\SSSSS^{(0)}$ and $\xi_2\in\se$, such that $\xi=\xi_1+\xi_2$
and $\xi_1(-\infty)=1$.
Set $A\define\{a\in P\mid \xi_1(a)>0\}$.
Then, since $\xi_1(-\infty)=1$ and $\xi_1\in\SSSSS^{(0)}$, we see that
$A$ is an antichain.
Let $C$ be an arbitrary element of $C_\xi^{[d-\epsilon]}$.
Since $\xi_2(-\infty)=d-1$ and $\xi_2\in\se$, we see that 
$\xip_2(C)+\epsilon\leq d-1$.
Since $\xip_1(C)+\xip_2(C)=\xip(C)=d-\epsilon$, we see that $\xip_1(C)\geq 1$.
Thus, $C\cap A\neq \emptyset$ and for the unique element $a\in A\cap C$,
$\xi(a)=\xi_1(a)+\xi_2(a)\geq1+\epsilon >\epsilon$
since $\xi_2\in\se$.
\end{proof}
By using this criterion of a Laurent monomial $T^\xi\in\omegae$ to be a generator
of $\omegae$, we next state a sufficient condition for a Laurent monomial 
$T^\xi\in\omegae$ to be a generator of $\omegae$.

\begin{lemma}
\mylabel{lem:gen suf}
Let $\xi\in\se$ and set $\xi(-\infty)=d$.
If there are elements $z_0$, $w_1$, $z_1$, $w_2$, \ldots, $z_{s-1}$, $w_s$
of $P$ with
$z_0>w_1<z_1>w_2<\cdots<z_{s-1}>w_s$ and
$C_0$, $C_1$, \ldots, $C_s\in C_\xi^{[d-\epsilon]}$ such that
\begin{enumerate}
\item
$s=0$ or
$C_0\ni z_0$, $C_s\ni w_s$ and
$C_i\ni w_i$, $z_i$ for $1\leq i\leq s-1$
(the latter part is trivially valid if $s=1$),  and
\item
$\xi(z)=\epsilon$ for any $z\in C_i\cap(w_i,z_i)$ and for any $i$ with 
$0\leq i\leq s$, where we set $w_0=-\infty$ and $z_s=\infty$,
\end{enumerate}
then $T^\xi$ is a generator of $\omegae$
\end{lemma}
\begin{proof}
Assume the contrary.
Then by Lemma \ref{lem:not a gen}, we see that there is an antichain $A$ in $P$ such that
for any $C\in C_\xi^{[d-\epsilon]}$, $A\cap C\neq\emptyset$ and $\xi(a)>\epsilon$ for the
unique element $a\in A\cap C$.

If $s=0$, then for $a\in A\cap C_0$, $\xi(a)=\epsilon$ by assumption and this is a contradiction.
Thus, $s>0$.
Set $A\cap C_i=\{a_i\}$ for $0\leq i\leq s$.
Since $a_s\leq w_s$ and $a_0\geq z_0$ by assumption, we see that there exists 
$i$ with $1\leq i\leq s$, $a_i\leq w_i$ and $a_{i-1}\geq z_{i-1}$.
Since $z_{i-1}>w_i$, this contradicts to the fact that $A$ is an antichain.
\end{proof}

Next we show a strong converse of this lemma.
We show that if $T^\xi$ is a generator of $\omegae$, then we can not only take a
sequence $z_0$, $w_1$, \ldots, $z_{s-1}$, $w_s$ satisfying the condition of 
Lemma \ref{lem:gen suf}, but also $z_0$, $w_1$, \ldots, $z_{s-1}$, $w_s$ satisfy \condn'.
Further, we show that we can take sequence which are ``\qered''.

\begin{definition}
\rm
\mylabel{def:qered}
For a sequence of elements $w_0$, $z_0$, $w_1$, \ldots, $z_{s-1}$, $w_s$, $z_s$
with 
$w_0<z_0>w_1<\cdots<z_{s-1}>w_s<z_s$ in $P^\pm$, we set
$$
\qe(w_0,z_0, \ldots, w_s, z_s)\define
\sum_{\ell=0}^s\qedist(w_\ell,z_\ell)-\sum_{\ell=0}^{s-1}\qedist(w_{\ell+1},z_\ell).
$$
Further, for a sequence
$y_0$, $x_1$, \ldots, $y_{t-1}$, $x_t$ of elements in $P$ with \condn',
we say that $y_0$, $x_1$, \ldots, $y_{t-1}$, $x_t$ is \qered\ if
$$
\qedist(x_i,y_j)<\qe(x_i,y_i,\ldots,x_j,y_j)
$$
for any $x_i$, $y_j$ with $x_i<y_j$ and $0\leq i<j\leq t$,
where we set $x_0=-\infty$ and $y_t=\infty$.
\end{definition}

\begin{example}
\rm
If there is a part of the sequence of the following form
$$
\begin{picture}(60,40)

\put(10,10){\circle*{3}}
\put(10,30){\circle*{3}}
\put(30,10){\circle*{3}}
\put(30,30){\circle*{3}}
\put(50,10){\circle*{3}}
\put(50,30){\circle*{3}}

\put(20,20){\circle*{3}}

\put(10,10){\line(2,1){40}}
\put(10,10){\line(0,1){20}}
\put(10,30){\line(1,-1){20}}
\put(30,10){\line(0,1){20}}
\put(30,30){\line(1,-1){20}}
\put(50,10){\line(0,1){20}}

\put(9,31){\makebox(0,0)[br]{$y_i$}}
\put(9,9){\makebox(0,0)[tr]{$x_i$}}
\put(30,9){\makebox(0,0)[t]{$x_{i+1}$}}
\put(30,31){\makebox(0,0)[b]{$y_{i+1}$}}
\put(51,9){\makebox(0,0)[tl]{$x_{i+2}$}}
\put(51,31){\makebox(0,0)[bl]{$y_{i+2}$}}

\end{picture}
\begin{picture}(10,40)
\put(5,20){\makebox(0,0){or}}
\end{picture}
\begin{picture}(40,40)

\put(10,10){\circle*{3}}
\put(10,30){\circle*{3}}
\put(30,10){\circle*{3}}
\put(30,30){\circle*{3}}


\put(10,10){\line(1,1){20}}
\put(10,10){\line(0,1){20}}
\put(10,30){\line(1,-1){20}}
\put(30,10){\line(0,1){20}}

\put(9,31){\makebox(0,0)[br]{$y_i$}}
\put(9,9){\makebox(0,0)[tr]{$x_i$}}
\put(31,9){\makebox(0,0)[t]{$x_{i+1}$}}
\put(31,31){\makebox(0,0)[b]{$y_{i+1}$}}

\end{picture}
$$
then it is not \qonered,
since
$\qonedist(x_i,y_{i+2})=1>0=\qone(x_i,y_i,x_{i+1},y_{i+1},x_{i+2},y_{i+2})$
in the former case and
$\qone\dist(x_i,y_{i+1})=1=\qone(x_i,y_i,x_{i+1},y_{i+1})$
in the latter case.
Later, we seek a 
sequence
$y_0$, $x_1$, \ldots, $y_{t-1}$, $x_t$ with
\condn'\ 
such that
$\qe(x_0, y_0, \ldots, x_t, \infty)$ as large as possible.
If there is a part of the first kind in the \scn', we can replace it with
$y_0$, $x_1$, \ldots, $x_i$, $y_{i+2}$, $x_{i+2}$, \ldots, $x_t$
and obtain a sequence with larger $\qone(x_0,y_0,\ldots)$.
Further, if there is a part of the second kind in the \scn', we apply 
the replacement above and remove redundancy.
\end{example}

\begin{lemma}
\mylabel{lem:gen nec}
Let $\xi$ be an element of $\se$ with $T^\xi$ is a generator of $\omegae$.
Set $d\define \xi(-\infty)$.
Then there exists a \qered\ sequence $y_0$, $x_1$, \ldots, $y_{t-1}$, $x_t$ with \condn'
such that there are $C_0$, $C_1$, \ldots, $C_t\in C_\xi^{[d-\epsilon]}$ with
the following conditions.
\begin{enumerate}
\item
$t=0$ or
$C_0\ni y_0$, $C_t\ni x_t$ and
$C_i\ni x_i$, $y_i$ for $1\leq i\leq t-1$,  and
\item
$\xi(z)=\epsilon$ for any $z\in C_i\cap(x_i,y_i)$ and for any $i$ with 
$0\leq i\leq t$, where we set $x_0=-\infty$ and $y_t=\infty$.
\end{enumerate}
\end{lemma}
\begin{proof}
First we note that $C_\xi^{[d-\epsilon]}\neq\emptyset$ by Lemma \ref{lem:not a gen}.
Further, if there exists $C\in C_\xi^{[d-\epsilon]}$ with $\xi(z)=\epsilon$ for any $z\in C$,
then the empty sequence satisfies the required condition.
Thus, in the following of the proof, we assume that for any $C\in C_\xi^{[d-\epsilon]}$,
there is $z\in C$ with $\xi(z)>\epsilon$.

Put
$Y_0\define\{y\in P\mid \exists C\in C_\xi^{[d-\epsilon]}; y=\min\{c\in C\mid \xi(c)>\epsilon\}\}$,
$X_1\define\{x\in P\mid \exists y\in Y_0; x<y\}$,
$Y_1\define Y_0\cup\{y\in P\mid \exists x\in X_1, \exists C\in C_\xi^{[d-\epsilon]}; 
x\in C, y=\min\{c\in C\mid c>x, \xi(c)>\epsilon\}\}$,
$X_2\define\{x\in P\mid \exists y\in Y_1; x<y\}$,
$Y_2\define Y_1\cup\{y\in P\mid \exists x\in X_2, \exists C\in C_\xi^{[d-\epsilon]}; 
x\in C, y=\min\{c\in C\mid c>x, \xi(c)>\epsilon\}\}$,
$X_3\define\{x\in P\mid \exists y\in Y_2; x<y\}$
and so on.
Also put $X\define\bigcup_{i=1}^\infty X_i$ and $Y\define\bigcup_{i=0}^\infty Y_i$.
Set $A\define \max Y$.
Then $A$ is an antichain in $P$.
Moreover, $Y\setminus A\subset X$.
In fact, if $y\in Y\setminus A$, there is $a\in A$ with $a>y$.
Since $a\in A\subset Y$, we see that $y\in X$.

If $A\cap C\neq\emptyset$ for any $C\in C_\xi^{[d-\epsilon]}$, then by Lemma \ref{lem:not a gen},
we see that $T^\xi$ is not a generator of $\omegae$, since $\xi(a)>\epsilon$ for any $a\in A$,
contradicting the assumption.
Therefore, there exists $C\in C_\xi^{[d-\epsilon]}$ such that $C\cap A=\emptyset$.

Take such $C$ and set $\{c\in C\mid \xi(c)>\epsilon\}=\{c_1,c_2,\ldots, c_v\}$,
$c_1<c_2<\cdots<c_v$.
By the definition of $Y_0$, we see that $c_1\in Y_0\subset Y$.
Since $C\cap A=\emptyset$, we see that $c_1\in Y\setminus A\subset X$.
Therefore, we see that $c_2\in Y$, and so on.
Repeating this argument, we see that $c_v\in X$.

Set $t\define\min\{\ell\mid c_v\in X_\ell\}$, $x_t\define c_v$ and $C_t\define C$.
Since $x_t\in X_t$, we see that there is $y_{t-1}\in Y_{t-1}$ with $y_{t-1}>x_t$.
Note that $y_{t-1}\not\in Y_{t-2}$ since $x_t\not\in X_{t-1}$.
By the definition of $Y_{t-1}$, we see that there are $x_{t-1}\in X_{t-1}$ and
$C_{t-1}\in C_{\xi}^{[d-\epsilon]}$ with 
$x_{t-1}\in C_{t-1}$ and
$y_{t-1}=\min\{c\in C_{t-1}\mid \xi(c)>\epsilon, c>x_{t-1}\}$.
Note that $x_{t-1}\not\in X_{t-2}$ since $y_{t-1}\not\in Y_{t-2}$.
Since $x_{t-1}\in X_{t-1}$, there is $y_{t-2}\in Y_{t-2}$ with $y_{t-2}>x_{t-1}$.
Note that $y_{t-2}\not\in Y_{t-3}$ since $x_{t-1}\not\in X_{t-2}$.

Continueing this argument, we see that there exist a sequence
$y_0$, $x_1$, \ldots, $y_{t-1}$, $x_t$ of elements of $P$ 
and $C_1$, $C_2$, \ldots, $C_t\in C_\xi^{[d-\epsilon]}$ such that
$y_0>x_1<\ldots<y_{t-1}>x_t$,
$x_k\in X_k$ for $1\leq k \leq t$, $x_k\not\in X_{k-1}$ for $2\leq k\leq t$,
$y_k\in Y_k$ for $0\leq k\leq t-1$, $y_k\not\in Y_{k-1}$ for $1\leq k\leq t-1$,
$y_k$, $x_k\in C_k$ and $\xi(z)=\epsilon$ for any $z\in C_k\cap(x_k,y_k)$ 
for $1\leq k\leq t-1$,
$x_t\in C_t$ and $\xi(z)=\epsilon$ for any $z\in C_t\cap(x_t,\infty)$.
Since $x_k\not\in X_{k-1}$ for $2\leq k\leq t$, we see that if $i\leq j-2$,
then $y_i\not> x_j$,
i.e., $y_0$, $x_1$, \ldots, $y_{t-1}$, $x_t$ satisfies \condn'.
Moreover, by the definition of $Y_0$, we can take $C_0\in C_\xi^{[d-\epsilon]}$ with
$y_0=\min\{c\in C_0\mid \xi(c)>\epsilon\}$.
Then $y_0\in C_0$ and for any $z\in C_0\cap(-\infty,y_0)$, $\xi(z)=\epsilon$.

Next we show that $y_0$, $x_1$, \ldots, $y_{t-1}$, $x_t$ is \qered.
Assume the contrary and suppose that there are $i$ and $j$ with $0\leq i<j\leq t$,
$x_i<y_j$ and
$\qedist(x_i,y_j)\geq
\sum_{\ell=i}^j\qedist(x_\ell,y_\ell)-\sum_{\ell=i}^{j-1}\qedist(x_{\ell+1},y_\ell)$.
Then
\begin{equation}
\qedist(x_i,y_j)+\sum_{\ell=i}^{j-1}\qedist(x_{\ell+1},y_\ell)
\geq
\sum_{\ell=i}^j\qedist(x_\ell,y_\ell).
\mylabel{eq:cross sum}
\end{equation}
Take a saturated chain 
$x_{\ell+1}=z_{\ell,0}\covered z_{\ell,1}\covered\cdots\covered z_{\ell,s_\ell}=y_\ell$
with $\epsilon s_\ell=\qedist(x_{\ell+1},y_\ell)$ for $i\leq \ell\leq j-1$.
Then,
$C'_\ell\define(C_{\ell+1}\cap(-\infty,x_{\ell+1}])\cup\{z_{\ell,1},\ldots, z_{\ell, s_{\ell}-1}\}
\cup(C_\ell\cap[y_\ell,\infty))$
is a maximal chain in $P$ for $i\leq \ell\leq j-1$.
Further, take a saturated chain
$x_i=z'_0\covered z'_1\covered\cdots\covered z'_{s'}=y_j$
with $\epsilon s'=\qedist(x_i,y_j)$.
Then,
$C''\define(C_i\cap(-\infty,x_i])\cup\{z'_1,\ldots,z'_{s'-1}\}\cup(C_j\cap[y_j,\infty))$
is also a maximal chain in $P$.

Therefore, 
\begin{eqnarray}
&&\qedist(x_i,y_j)+\sum_{\ell=i}^{j-1}\qedist(x_{\ell+1},y_\ell)
\nonumber\\
&&\qquad
+\sum_{\ell=i}^j\Big(\xip(C_\ell\cap(-\infty,x_\ell])+\xip(C_\ell\cap[y_\ell,\infty))\Big)
\nonumber\\
&=&
(s'+\sum_{\ell=i}^{j-1}s_\ell)\epsilon
+\sum_{\ell=i}^j\Big(\xip(C_\ell\cap(-\infty,x_\ell])+\xip(C_\ell\cap[y_\ell,\infty))\Big)
\nonumber\\
&=&
((s'-1)+\sum_{\ell=i}^{j-1}(s_\ell-1))\epsilon+(j-i+1)\epsilon
\nonumber\\
&&\qquad
+\sum_{\ell=i}^j\Big(\xip(C_\ell\cap(-\infty,x_\ell])+\xip(C_\ell\cap[y_\ell,\infty))\Big)
\nonumber\\
&\leq&
\sum_{\ell=i}^{j-1}\Big(\xip(C_{\ell+1}\cap(-\infty,x_{\ell+1}])+
\sum_{m=1}^{s_\ell-1}\xi(z_{\ell,m})
+\xip(C_\ell\cap[y_\ell,\infty))\Big)
\nonumber\\
&&\qquad
+\xip(C_i\cap(-\infty, x_i])+\sum_{m=1}^{s'-1}\xi(z'_m)+\xip(C_j\cap[y_j,\infty))
\nonumber\\
&&\qquad
+(j-i+1)\epsilon
\mylabel{eq:cross sum xi}
\\
&=&
\sum_{\ell=i}^{j-1}\xip(C'_\ell)+\xip(C'')+\epsilon(j-i+1)
\nonumber\\
&\leq&
(j-i+1)\xi(-\infty)=d(j-i+1),
\mylabel{eq:cross sum xip}
\end{eqnarray}
since $C'_{\ell}$ for $i\leq \ell\leq j-1$ and $C''$ are maximal chains in $P$
and $\xi\in\SSSSS^{(\epsilon)}$.
On the other hand,
\begin{eqnarray*}
&&d(j-i+1)\\
&=&\sum_{\ell=i}^j\xip(C_\ell)+\epsilon(j-i+1)\\
&=&
\sum_{\ell=i}^j\Big(\xip(C_\ell\cap(-\infty,x_\ell])+\xip(C_\ell\cap(x_\ell,y_\ell))
+\xip(C_\ell\cap[y_\ell,\infty))+\epsilon\Big)\\
&=&
\sum_{\ell=i}^j\Big(\xip(C_\ell\cap(-\infty,x_\ell])+\xip(C_\ell\cap[y_\ell,\infty))
+(\#(C_\ell\cap(x_\ell,y_\ell))+1)\epsilon\Big)\\
&\leq&
\sum_{\ell=i}^j\Big(\xip(C_\ell\cap(-\infty,x_\ell])+\xip(C_\ell\cap[y_\ell,\infty)
+\qedist(x_\ell,y_\ell)\Big),
\end{eqnarray*}
since 
for any $i\leq \ell\leq j$, $C_\ell\in C_\xi^{[d-\epsilon]}$ and
$\xi(z)=\epsilon$ for any $z\in C_\ell\cap(x_\ell,y_\ell)$.
Thus, we see by inequation \refeq{eq:cross sum} that equalities hold in \refeq{eq:cross sum xi}
and \refeq{eq:cross sum xip}
and therefore $\xi(z'_m)=\epsilon$ for $1\leq m\leq s'-1$
and $\xip(C'')=d-\epsilon$.
This means that $y_j\in Y_i$, contradicting the fact that $y_j\not\in Y_{j-1}$.

Therefore, we see that $y_0$, $x_1$, \ldots, $y_{t-1}$, $x_t$ is \qered.
\end{proof}
%
%

By Lemmas \ref{lem:gen suf} and \ref{lem:gen nec}, we see the following.

\begin{prop}
\mylabel{prop:gen equiv}
Let $\xi\in\se$ and set $d\define\xi(-\infty)$.
Then the following conditions are equivalent.
\begin{enumerate}
\item
$T^\xi$ is a genenrator of $\omegae$.
\item
There are elements $z_0$, $w_1$, \ldots, 
$z_{s-1}$, $w_s$
of $P$ with 
$z_0>w_1<\cdots<z_{s-1}>w_s$ 
and $C_0$, $C_1$, \ldots, $C_s\in C_\xi^{[d-\epsilon]}$ such that
\begin{enumerate}
\item
$s=0$ or
$C_0\ni z_0$, $C_s\ni w_s$ and
$C_i\ni w_i$, $z_i$ for $1\leq i\leq s-1$,  and
\item
$\xi(z)=\epsilon$ for any $z\in C_i\cap(w_i,z_i)$ and for any $i$ with 
$0\leq i\leq s$, where we set $w_0=-\infty$ and $z_s=\infty$,
\end{enumerate}
\item
\mylabel{item:qered}
There exists a \qered\ sequence $y_0$, $x_1$, \ldots, $y_{t-1}$, $x_t$ with \condn'
such that there are $C_0$, $C_1$, \ldots, $C_t\in C_\xi^{[d-\epsilon]}$
with the following conditions.
\begin{enumerate}
\item
$t=0$ or
$C_0\ni y_0$, $C_t\ni x_t$ and
$C_i\ni x_i$, $y_i$ for $1\leq i\leq t-1$,  and
\item
$\xi(z)=\epsilon$ for any $z\in C_i\cap(x_i,y_i)$ and for any $i$ with 
$0\leq i\leq t$, where we set $x_0=-\infty$ and $y_t=\infty$.
\end{enumerate}
\end{enumerate}
\end{prop}
As a corollary, we obtain an upper bound of the degrees of generators of $\omegae$.

\begin{cor}
\mylabel{cor:gen deg max}
Let $T^\xi$, $\xi\in\se$, be a generator of $\omegae$.
Then
$$
\deg T^\xi\leq\max\left\{\qe(x_0,y_0,\ldots,x_t,y_t)\left|
\vcenter{\hsize=.4\textwidth\relax
\noindent
$y_0$, $x_1$, \ldots, $y_{t-1}$, $x_t$ is a \qered\ \scn',
where we set $x_0\define-\infty$ and $y_t\define\infty$.}
\right.\right\}.
$$
(See Definition \ref{def:qered} for the notation.)
\end{cor}
\begin{proof}
Set $d\define \deg T^\xi$.
Then by Proposition \ref{prop:gen equiv}, we see that there are a \qered\ sequence
$y_0$, $x_1$, \ldots, $y_{t-1}$, $x_t$ with \condn' and $C_0$, $C_1$, \ldots, 
$C_t\in C_\xi^{[d-\epsilon]}$ satisfying 
\ref{item:qered} of Proposition \ref{prop:gen equiv}.

If $t=0$, then since $\xi(c)=\epsilon$ for any $c\in C_0$, we see that
$$
d-\epsilon=\xip(C_0)=(\#C_0)\epsilon\leq\qedist(-\infty,\infty)-\epsilon.
$$
Therefore, $\deg T^\xi=d\leq\qedist(-\infty,\infty)=\qe(x_0,y_0)$.

Now suppose that $t>0$.
Since $\max\{\xip(C)\mid C$ is a maximal chain in $P\}=d-\epsilon$,
we see that
$
d-\epsilon\geq\xipdown(x_i)+\xippup(x_i)\geq
\sum_{c\in C_i\cap(-\infty,x_i]}\xi(c)+\sum_{c\in C_i\cap(x_i,\infty)}\xi(c)
=\xip(C_i)=d-\epsilon$.
Therefore,
$$
\xipdown(x_i)=\sum_{c\in C_i\cap(-\infty,x_i]}\xi(c)
\quad\mbox{and}\quad
\xippup(x_i)=\sum_{c\in C_i\cap(x_i,\infty)}\xi(c)
$$
for $1\leq i\leq t$.
We see by the same way that
$$
\xippdown(y_i)=\sum_{c\in C_i\cap(-\infty,y_i)}\xi(c)
\quad\mbox{and}\quad
\xipup(y_i)=\sum_{c\in C_i\cap[y_i,\infty)}\xi(c)
$$
for $0\leq i\leq t-1$.
In particular,
$$
\xippup(x_i)-\xipup(y_i)=\sum_{c\in C_i\cap(x_i,y_i)}\xi(c)=(\#(C_i\cap(x_i,y_i)))\epsilon
\leq \qedist(x_i,y_i)-\epsilon
$$
for $1\leq i\leq t-1$, since $\xi(c)=\epsilon$ for any $c\in C_i\cap(x_i,y_i)$.
Moreover, we see that
$$
\xippdown(y_0)\leq\qedist(-\infty,y_0)-\epsilon
\quad\mbox{and}\quad
\xippup(x_t)\leq\qedist(x_t,\infty)-\epsilon
$$
by the same way.
On the other hand, we see by Lemma \ref{lem:plus ineq},
$$
\xippup(x_{i+1})-\xipup(y_i)\geq\qedist(x_{i+1},y_i)-\epsilon
$$
for $0\leq i\leq t-1$.
Thus,
\begin{eqnarray*}
&&\qe(x_0,y_0,\ldots,x_t,y_t)\\
&=&\sum_{\ell=0}^t\qedist(x_\ell,y_\ell)-
\sum_{\ell=0}^{t-1}\qedist(x_{\ell+1},y_\ell)\\
&\geq&
\xippdown(y_0)+\epsilon+\sum_{\ell=1}^{t-1}(\xippup(x_\ell)-\xipup(y_\ell)+\epsilon)
+\xippup(x_t)+\epsilon\\
&&\qquad
-\sum_{\ell=0}^{t-1}(\xippup(x_{\ell+1})-\xipup(y_\ell)+\epsilon)
\\
&=&
\xippdown(y_0)+\xipup(y_0)+\epsilon\\
&=&
\xip(C_0)+\epsilon\\
&=&d.
\end{eqnarray*}
\end{proof}
In fact, this upper bound is the least upper bound, see Proposition \ref{prop:gen const}.
Further, it is 
seen, by Corollary \ref{cor:gen deg min} 
that the minimum degree of the 
Laurent monomials in $\omegae$ is
$\qedist(-\infty,\infty)$.

The \qered\ \scn' satisfying \ref{item:qered} of Proposition \ref{prop:gen equiv}
is far from unique.
However, for a given \qered\ sequence $y_0$, $x_1$, \ldots, $y_{t-1}$, $x_t$ with
\condn'
we can construct an element $\xi\in\se$ such
that 
$d=\xi(-\infty)=\qe(-\infty,y_0, \ldots, x_t,\infty)$
and
there are $C_0$, $C_1$, \ldots, $C_t\in C_\xi^{[d-\epsilon]}$ which satisfy
conditions of \ref{item:qered} of Proposition \ref{prop:gen equiv}.
In order to do this task, suppose that a \qered\ sequence
$y_0$, $x_1$, \ldots, $y_{t-1}$, $x_t$ with \condn' is given and fixed.
We set $x_0\define-\infty$, $y_t\define\infty$
and $d\define\qe(x_0,y_0,\ldots,x_t,y_t)$.

If $t=0$, then it is enough to set
$$
\xi(z)\define
\begin{cases}
\epsilon,&\mbox{if $z\in P$;}\\
d,&\mbox{if $z=-\infty$},
\end{cases}
$$
since 
$\xi\in\se$ by Lemma \ref{lem:an sn elem} and
there is a maximal chain $C_0$ in $P$ with 
$\xip(C_0)=\epsilon(\#C_0)=\qedist(-\infty,\infty)-\epsilon
=d-\epsilon$.
Thus, we assume that $t>0$ in the following.

We define two maps $\mu'\colon\{y_0,y_1,\ldots, y_t\}\to\ZZZ$
and $\mu''\colon\{x_0,x_1,\ldots, x_t\}\to\ZZZ$
by
$\mu''(x_i)\define\qe(x_i,y_i,\ldots,x_t,y_t)$
for $0\leq i\leq t$,
$\mu'(y_i)\define\qe(x_{i+1},y_{i+1},\ldots,x_t,y_t)-\qedist(x_{i+1},y_i)$
for $0\leq i\leq t-1$ and $\mu'(y_t)\define0$.
Note that 
$\mu''(x_0)=d$,
$\mu''(x_i)=\mu'(y_i)+\qedist(x_i,y_i)$ for $0\leq i\leq t$,
$\mu''(x_i)=\mu'(y_{i-1})+\qedist(x_{i},y_{i-1})$ for $1\leq i\leq t$
and
$\mu''(x_i)-\mu'(y_j)=\qe(x_i,y_i,\ldots,x_j,y_j)$ for $0\leq i\leq j\leq t$.

Moreover, if $i<j$ and $x_i<y_j$, then
$\qedist(x_i,y_j)<\mu''(x_i)-\mu'(y_j)$,
since $y_0$, $x_1$, \ldots, $y_{t-1}$, $x_t$ is \qered.
Note that it may happen that $x_i=y_j$ for some $i$ and $j$,
but this does not imply $\mu''(x_i)=\mu'(y_j)$.

\begin{example}
\rm
Let $\epsilon=1$ and let $P$ and $y_0$, $x_1$, $y_1$, $x_2$ be as follows.
$$
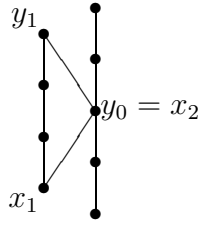
\begin{picture}(40,60)

\put(10,15){\circle*{2}}
\put(10,25){\circle*{2}}
\put(10,35){\circle*{2}}
\put(10,45){\circle*{2}}

\put(20,10){\circle*{2}}
\put(20,20){\circle*{2}}
\put(20,30){\circle*{2}}
\put(20,40){\circle*{2}}
\put(20,50){\circle*{2}}

\put(9,14){\makebox(0,0)[tr]{$x_1$}}
\put(9,46){\makebox(0,0)[br]{$y_1$}}
\put(21,30){\makebox(0,0)[l]{$y_0=x_2$}}

\put(10,15){\line(0,1){30}}
\put(20,10){\line(0,1){40}}
\put(10,15){\line(2,3){10}}
\put(10,45){\line(2,-3){10}}

\end{picture}
$$
Then $\mu'(y_2)=0$, $\mu''(x_2)=3$, $\mu'(y_1)=2$, $\mu''(x_1)=5$, $\mu'(y_0)=4$
and $\mu''(x_0)=7$.
\end{example}

We first define the map $\xi_0\colon P\to\ZZZ$ by
$$
\xi_0(z)\define
\begin{cases}
\mu'(z)-\max\{\mu'(y_j)+\qedist(z,y_j)\mid y_j>z\mbox{ in $P^+$}\}+\epsilon,&
\mbox{if $z=y_i$ for some $i$,}\\
\epsilon,&\mbox{otherwise.}
\end{cases}
$$
Note that $y_j$ may be $y_t=\infty$ in the above definition,
thus $\{\mu'(y_j)+\qedist(z,y_j)\mid y_j>z$ in $P^+\}\neq\emptyset$,
since $z\in P$.
For $\xi_0$, we see the following fact.

\begin{lemma}
\mylabel{lem:xi0}
\begin{enumerate}
\item
\mylabel{item:gee}
$\xi_0(z)\geq\epsilon$ for any $z\in P$.
\item
\mylabel{item:eqe}
$\xi_0(z)=\epsilon$ if $z\not\in\{y_0,\ldots, y_{t-1}\}$.
\item
\mylabel{item:eqpup}
$\xipup_0(y_i)=\mu'(y_i)$ for $0\leq i\leq t-1$.
\item
\mylabel{item:xip}
$\xip_0(C)\leq d-\epsilon$ for any maximal chain $C$ in $P$.
\item
\mylabel{item:x0y0}
$\qedist(x_0,y_0)+\xipup_0(y_0)= d$.
\end{enumerate}
\end{lemma}
\begin{proof}
\ref{item:eqe} is obvious from the definition.
We next prove \ref{item:gee}.
The case where $z\not\in\{y_0,\ldots, y_t\}$ follows from \ref{item:eqe}.
Assume that $z=y_i$ and $0\leq i\leq t-1$.
Take $j$ with $y_j>y_i$ and $\xi_0(y_i)=\mu'(y_i)-\big(\mu'(y_j)+\qedist(y_i,y_j)\big)+\epsilon$.
Then $i<j$ by Lemma \ref{lem:cond np basic} and $x_i<y_j$.
Since $y_0$, $x_1$, \ldots, $y_{t-1}$, $x_t$ is \qered\ \scn',
we see that
$$
\qedist(x_i,y_j)<\mu''(x_i)-\mu'(y_j)=\qedist(x_i,y_i)+\mu'(y_i)-\mu'(y_j).
$$
On the other hand, since
$$
\qedist(x_i,y_i)+\qedist(y_i,y_j)\leq\qedist(x_i,y_j),
$$
we see that
$$
\qedist(y_i,y_j)<\mu'(y_i)-\mu'(y_j).
$$
Therefore, 
$$
\xi_0(y_i)=\mu'(y_i)-\big(\mu'(y_j)+\qedist(y_i,y_j)\big)+\epsilon
>\epsilon.
$$

Next we prove \ref{item:eqpup} by backward 
induction on $i$.
First consider the case where $i=t-1$.
Since $y_0$, $x_1$, \ldots, $y_{t-1}$, $x_t$ satisfies \condn', we see that
$y_j>y_{t-1}$ implies $j=t$ by Lemma \ref{lem:cond np basic}.
Therefore,
\begin{eqnarray*}
\xi_0(y_{t-1})&=&\mu'(y_{t-1})-\big(\mu'(y_t)+\qedist(y_{t-1},y_t)\big)+\epsilon\\
&=&\mu'(y_{t-1})-\qedist(y_{t-1},y_t)+\epsilon.
\end{eqnarray*}
On the other hand, since
$z\in P$, $z>y_{t-1}$ implies $z\not\in\{y_0,\ldots, y_{t-1}\}$ 
by Lemma \ref{lem:cond np basic}, we see that
$$
\xipup_0(y_{t-1})=\xi_0(y_{t-1})+\qedist(y_{t-1},y_t)-\epsilon.
$$
Therefore, we see that
$$
\xipup_0(y_{t-1})=\mu'(y_{t-1}).
$$
Next suppose that $0\leq i\leq t-2$.
We first show that $\xipup_0(y_i)\leq\mu'(y_i)$.
Take $z_0$, $z_1$, \ldots, $z_s$ with 
$y_i=z_0\covered z_1\covered\cdots\covered z_s\covered\infty$
and
$\sum_{\ell=0}^s\xi_0(z_\ell)=\xipup(y_i)$.
First consider the case where 
$\{z_1$, \ldots, $z_s\}\cap\{y_0$, \ldots, $y_{t-1}\}=\emptyset$.
Since $\xi_0(z_\ell)=\epsilon$ for $1\leq \ell\leq s$, we see that
\begin{eqnarray*}
\xipup_0(y_i)&=&\xi_0(y_i)+\sum_{\ell=1}^s\xi_0(z_\ell)\\
&=&\xi_0(y_i)+(s+1)\epsilon-\epsilon\\
&\leq&\xi_0(y_i)+\qedist(y_i,\infty)-\epsilon\\
&=&\mu'(y_i)-\max\{\mu'(y_j)+\qedist(y_i,y_j)\mid y_j>y_i\}+\epsilon\\
&&\qquad+\qedist(y_i,y_t)-\epsilon\\
&\leq&\mu'(y_i).
\end{eqnarray*}
If $\{z_1, \ldots, z_s\}\cap\{y_0, \ldots, y_{t-1}\}\neq\emptyset$,
take minimal $u$ with $u\geq 1$ and $z_u\in\{y_0, \ldots, y_{t-1}\}$
and set $z_u=y_j$.
Then $j>i$ by Lemma \ref{lem:cond np basic}.
Thus, $\xipup_0(y_j)=\mu'(y_j)$ by the induction hypothesis.
Moreover, 
$\xipup_0(y_j)=\sum_{\ell=u}^s\xi_0(z_\ell)$ by Lemma \ref{lem:path} \ref{item:plus path}.
On the other hand, since $\xi_0(z_\ell)=\epsilon$ for $1\leq \ell\leq u-1$,
we see that 
$$\sum_{\ell=1}^{u-1}\xi_0(z_\ell)=u\epsilon -\epsilon\leq\qedist(y_i,y_j)-\epsilon.
$$
Therefore,
\begin{eqnarray*}
\sum_{\ell=1}^s\xi_0(z_\ell)
&\leq&\qedist(y_i,y_j)-\epsilon+\xipup_0(y_j)\\
&=&\mu'(y_j)+\qedist(y_i,y_j)-\epsilon.
\end{eqnarray*}
Thus,
\begin{eqnarray*}
\xipup_0(y_i)
&=&\xi_0(y_i)+\sum_{\ell=1}^s\xi_0(z_\ell)\\
&\leq&\mu'(y_i)-\max\{\mu'(y_{\jmath'})+\qedist(y_i,y_{\jmath'})\mid y_{\jmath'}>y_i\}+\epsilon\\
&&\qquad+\mu'(y_j)+\qedist(y_i,y_j)-\epsilon\\
&\leq&\mu'(y_i).
\end{eqnarray*}

Next we prove that $\xipup_0(y_i)\geq\mu'(y_i)$.
Take $j$ with $y_j>y_i$ and 
\begin{equation}
\xi_0(y_i)=\mu'(y_i)-\big(\mu'(y_j)+\qedist(y_i,y_j)\big)+\epsilon.
\mylabel{eq:xi yi def}
\end{equation}
Then $j>i$ by Lemma \ref{lem:cond np basic}.
Take $z_0$, \ldots, $z_u\in P^+$ with
$y_i=z_0\covered\cdots\covered z_u=y_j$
and $u\epsilon=\qedist(y_i,y_j)$.
Then by Lemma \ref{lem:path} \ref{item:chain path}
and the induction hypothesis
$$
\xipup_0(y_i)
\geq\sum_{\ell=0}^{u-1}\xi_0(z_\ell)+\xipup_0(y_j)
=\sum_{\ell=0}^{u-1}\xi_0(z_\ell)+\mu'(y_j),
$$
where we set $\xipup_0(y_t)\define0$.
Since $\xi_0(z_\ell)\geq\epsilon$ for $1\leq \ell\leq u-1$ by \ref{item:gee}, 
we see that
\begin{eqnarray*}
\xipup_0(y_i)
&\geq&\xi_0(y_i)+u\epsilon-\epsilon +\mu'(y_j)\\
&=&\xi_0(y_i)+\qedist(y_i,y_j)-\epsilon+\mu'(y_j)\\
&=&\mu'(y_i)
\end{eqnarray*}
by \refeq{eq:xi yi def}.
Thus, we have proved \ref{item:eqpup}.

\ref{item:x0y0} follows from \ref{item:eqpup} and the fact 
$ d=\mu''(x_0)=\mu'(y_0)+\qedist(x_0,y_0)$.

Finally, we prove \ref{item:xip}.
Let $C$ be a maximal chain in $P$.
We set $C=\{c_1,\ldots, c_s\}$,
$-\infty\covered c_1\covered\cdots\covered c_s\covered\infty$.
First, consider the case where $C\cap\{y_0,\ldots, y_{t-1}\}=\emptyset$.
Then $\xi_0(c_i)=\epsilon$ for any $i$ by \ref{item:eqe}.
Thus,
$$
\xip_0(C)
=s\epsilon
=(s+1)\epsilon-\epsilon
\leq\qedist(-\infty,\infty)-\epsilon.
$$
Since $y_0$, $x_1$, \ldots $y_{t-1}$, $x_t$ is \qered, we see that
$$
\qedist(-\infty,\infty)=\qedist(x_0,y_t)< \qe(x_0,y_0,\ldots,x_t,y_t)=d.
$$
Therefore,
$$
\xip_0(C)\leq d-\epsilon.
$$
Next we consider the case where $C\cap\{y_0,\ldots, y_{t-1}\}\neq\emptyset$.
Take minimal $u$ with $c_u\in\{y_0, \ldots, y_{t-1}\}$ and set $c_u=y_j$.
Then 
$$
\xip_0(C)=\sum_{\ell=1}^{u-1}\xi_0(c_\ell)+\sum_{\ell=u}^s\xi_0(c_\ell).
$$
Since $\xi_0(c_\ell)=\epsilon$ for $1\leq \ell\leq u-1$,
we see that
$$
\sum_{\ell=1}^{u-1}\xi_0(c_\ell)=u\epsilon-\epsilon\leq\qedist(x_0,y_j)-\epsilon
\leq\mu''(x_0)-\mu'(y_j)-\epsilon
=d-\mu'(y_j)-\epsilon,
$$
since $y_0$, $x_1$, \ldots, $y_{t-1}$, $x_t$ is \qered.
On the other hand, since 
$\xipup_0(y_j)\geq\sum_{\ell=u}^s\xi_0(c_\ell)$, we see that
$$
\xip(C)\leq d-\mu'(y_j)-\epsilon+\xipup_0(y_j)
= d-\epsilon
$$
by \ref{item:eqpup}.
\end{proof}

Next we state a lemma which is used in the induction argument.

\begin{lemma}
\mylabel{lem:ind lemma}
Let $k$ be an integer with $0\leq k\leq t-1$ and suppose that a map
$\xi_k\colon P\to\ZZZ$ is defined so that
\begin{trivlist}
\item[\hskip\labelsep\bf(H1)]
$\xi_k(z)\geq\epsilon$ for any $z\in P$.
\item[\hskip\labelsep\bf(H2)]
$\xi_k(z)=\epsilon$ if $z\not\in\{x_1, \ldots, x_k,y_1,\ldots, y_{t-1}\}$.
\item[\hskip\labelsep\bf(H3)]
$\xipup_k(y_i)=\mu'(y_i)$ for $0\leq i\leq t-1$.
\item[\hskip\labelsep\bf(H4)]
$\xip_k(C)\leq d-\epsilon$ for any maximal chain $C$ in $P$.
\item[\hskip\labelsep\bf(H5)]
$\xipdown_k(x_i)+\qedist(x_i,y_i)+\xipup_k(y_i)= d$
for $0\leq i\leq k$.
\end{trivlist}
are satisfied, where we set $\xipdown_k(x_0)=0$.

Let $\xi_{k+1}\colon P\to\ZZZ$ be a map such that
$$
\xi_{k+1}(z)\define
\begin{cases}
\xi_k(z),&\mbox{if $z\neq x_{k+1}$;}\\
 d-\epsilon-\xippdown_{k}(x_{k+1})-\xippup_k(x_{k+1}),&
\mbox{if $z=x_{k+1}$.}
\end{cases}
$$
Then
\begin{enumerate}
\item
\mylabel{item:geek}
$\xi_{k+1}(z)\geq\xi_k(z)$ for any $z\in P$.
In particular,
$\xi_{k+1}(z)\geq\epsilon$ for any $z\in P$.
\item
\mylabel{item:eqek}
$\xi_{k+1}(z)=\epsilon$ if $z\not\in\{x_1, \ldots, x_{k+1}, y_0,\ldots, y_{t-1}\}$.
\item
\mylabel{item:eqpupk}
$\xipup_{k+1}(y_i)=\mu'(y_i)$ for $0\leq i\leq t-1$.
\item
\mylabel{item:xipk}
$\xip_{k+1}(C)\leq d-\epsilon$ for any maximal chain $C$ in $P$.
\item
\mylabel{item:x0y0k}
$\xipdown_{k+1}(x_i)+\qedist(x_i,y_i)+\xipup_{k+1}(y_i)= d$
for $0\leq i\leq k+1$,
where we set $\xipdown_{k+1}(x_0)=\xipup_{k+1}(y_t)=0$.
\end{enumerate}
\end{lemma}
\begin{proof}
First note that we have to be careful because that it might happen
$x_{k+1}\in \{y_0, \ldots, y_{t-1}\}$.

\ref{item:eqek} is obvious from (H2) and the definition of $\xi_{k+1}$.
Next we prove \ref{item:geek}.
Since
$$
\xippdown_k(x_{k+1})+\xi_k(x_{k+1})+\xippup_k(x_{k+1})\leq  d-\epsilon
$$
by (H4) and Lemma \ref{lem:plus ineq}, we see that
$$
\xi_{k+1}(x_{k+1})
= d-\epsilon-\xippdown_k(x_{k+1})-\xippup_k(x_{k+1})
\geq \xi_k(x_{k+1})
$$
by (H1).
Therefore \ref{item:geek} follows from the definition of $\xi_{k+1}$.

Next we prove \ref{item:xipk}.
Let $C$ be an arbitrary maximal chain in $P$.
If $C\not\ni x_{k+1}$, then by the definition of $\xi_{k+1}$, we see that
$$
\xip_{k+1}(C)=\xip_k(C)\leq d-\epsilon
$$
by (H4).
If $C\ni x_{k+1}$, then
\begin{eqnarray*}
\xip_{k+1}(C)
&\leq&\xippdown_{k+1}(x_{k+1})+\xi_{k+1}(x_{k+1})+\xippup_{k+1}(x_{k+1})\\
&=&\xippdown_{k}(x_{k+1})
+ (d-\epsilon-\xippdown_k(x_{k+1})-\xippup_k(x_{k+1}))+\xippup_k(x_{k+1})\\
&=& d-\epsilon
\end{eqnarray*}
by the definition of $\xi_{k+1}$.

Now we prove \ref{item:eqpupk}.
First, 
we see by \ref{item:geek} and (H3) that
$$
\xipup_{k+1}(y_i)\geq\xipup_k(y_i)=\mu'(y_i).
$$
In order to prove the converse inequality, take a saturated chain
$y_i=z_0\covered\cdots\covered z_s\covered\infty$ from $y_i$ to $\infty$
with
$\xipup_{k+1}(y_i)=\sum_{\ell=0}^s\xi_{k+1}(z_\ell)$.
If $x_{k+1}\not\in\{z_0,\ldots, z_{s}\}$, then
$$
\xipup_{k+1}(y_i)=\sum_{\ell=0}^s\xi_{k+1}(z_\ell)=\sum_{\ell=0}^s\xi_k(z_\ell)
\leq\xipup_k(y_i)=\mu'(y_i)
$$
by (H3) and the definition of $\xi_{k+1}$.
Now suppose that $x_{k+1}\in\{z_0, \ldots, z_s\}$ and set $x_{k+1}=z_u$.
Then, since $y_k>x_{k+1}\geq y_i>x_{i+1}$, we see by \condn' that
$k\not\leq (i+1)-2$, i.e., $k\geq i$.
(In fact, $i\leq k-1$, but we do not use this fact.)
Since $\xi_k(z)\geq\epsilon$ for any $z\in P$, we see by Lemma \ref{lem:plus ineq} that
$$
\xippdown_k(y_i)\geq\qedist(x_i,y_i)-\epsilon+\xipdown_k(x_i).
$$
Therefore, by (H5), \ref{item:geek} and \ref{item:xipk}, we see that
\begin{eqnarray*}
&& d-\epsilon\\
&=&\xipup_k(y_i)+\qedist(x_i,y_i)+\xipdown_k(x_i)-\epsilon\\
&\leq&\xipup_k(y_i)+\xippdown_k(y_i)\\
&\leq&\xipup_{k+1}(y_i)+\xippdown_{k+1}(y_i)\\
&\leq& d-\epsilon.
\end{eqnarray*}
Thus, we see that 
$$
\xipup_{k+1}(y_i)=\xipup_k(y_i)=\mu'(y_i).
$$

Finally, we prove \ref{item:x0y0k}.
By \ref{item:xipk} and Lemma \ref{lem:plus ineq}, we see that
$$
\xipup_{k+1}(y_i)+\qedist(x_i,y_i)-\epsilon+\xipdown_{k+1}(x_i)
\leq  d-\epsilon.
$$
Thus, 
$$
\xipup_{k+1}(y_i)+\qedist(x_i,y_i)+\xipdown_{k+1}(x_i)\leq d
$$
for any $i$.
If $ i\leq k$, then by (H5), we see that 
\begin{eqnarray*}
&&\xipup_{k+1}(y_i)+\qedist(x_i,y_i)+\xipdown_{k+1}(x_i)\\
&\geq&\xipup_k(y_i)+\qedist(x_i,y_i)+\xipdown_k(x_i)\\
&=& d.
\end{eqnarray*}
In order to prove the case where $i=k+1$, take a saturated chain
$x_{k+1}=z_0\covered\cdots \covered z_s\covered z_{s+1}=\infty$
with
$\xipup_{k+1}(x_{k+1})=\sum_{\ell=0}^s\xi_{k+1}(z_\ell)$.
Take minimal $u$ with $1\leq u\leq s+1$ and $z_u\in\{y_0, \ldots, y_t\}$
and set $z_u=y_j$.
Then
\begin{eqnarray}
\xippup_{k+1}(x_{k+1})
&=&(u-1)\epsilon+\sum_{\ell=u}^s\xi_{k+1}(z_\ell)
\nonumber\\
&\leq&\qedist(x_{k+1},y_j)-\epsilon+\xipup_{k+1}(y_j)
\nonumber\\
&=&\qedist(x_{k+1},y_j)-\epsilon+\mu'(y_j)
\mylabel{eq:qe mp}
\end{eqnarray}
by \ref{item:eqpupk} and Lemma \ref{lem:path} \ref{item:plus path}.
We claim here that $j=k$ or $j=k+1$.
In fact, by \condn', we see that $j\geq k$.
If $j\geq k+2$, then, since $y_0$, $x_1$, \ldots, $y_{t-1}$, $x_t$ is \qered, 
we see that
$$
\qedist(x_{k+1},y_j)<\mu''(x_{k+1})-\mu'(y_j).
$$
Thus,
by \refeq{eq:qe mp}, \ref{item:eqpupk} and Lemma \ref{lem:plus ineq}, we see that
\begin{eqnarray*}
\xippup_{k+1}(x_{k+1})
&<&\mu''(x_{k+1})-\epsilon\\
&=&\qedist(x_{k+1},y_{k+1})+\mu'(y_{k+1})-\epsilon\\
&=&\qedist(x_{k+1},y_{k+1})-\epsilon+\xipup_{k+1}(y_{k+1})\\
&\leq&\xippup_{k+1}(x_{k+1}).
\end{eqnarray*}
This is a contradiction.
Therefore, $j=k$ or $j=k+1$ and
$$
\mu''(x_{k+1})=\qedist(x_{k+1},y_j)+\mu'(y_j).
$$
Therefore, we see by \refeq{eq:qe mp} that
\begin{eqnarray*}
\xippup_{k+1}(x_{k+1})
&\leq&\mu''(x_{k+1})-\epsilon\\
&=&\mu'(y_{k+1})+\qedist(x_{k+1},y_{k+1})-\epsilon\\
&=&\xipup_{k+1}(y_{k+1})+\qedist(x_{k+1},y_{k+1})-\epsilon
\end{eqnarray*}
by \ref{item:eqpupk}.
Further, since 
$\xi_{k+1}(z)\geq\epsilon$ for any $z$, we see that 
$\xippup_{k+1}(x_{k+1})\geq\xipup_{k+1}(y_{k+1})+\qedist(x_{k+1},y_{k+1})-\epsilon$
by Lemma \ref{lem:path} \ref{item:chain path}.
Thus, 
$$\xippup_{k+1}(x_{k+1})=\xipup_{k+1}(y_{k+1})+\qedist(x_{k+1},y_{k+1})-\epsilon.$$
Since
$$
\xipdown_{k+1}(x_{k+1})+\xippup_{k+1}(x_{k+1})= d-\epsilon
$$
by the definition of $\xi_{k+1}$, we see that
\begin{eqnarray*}
&&\xipdown_{k+1}(x_{k+1})+\qedist(x_{k+1},y_{k+1})-\epsilon+\xipup_{k+1}(y_{k+1})\\
&=&\xipdown_{k+1}(x_{k+1})+\xippup_{k+1}(x_{k+1})\\
&=& d-\epsilon
\end{eqnarray*}
and therefore,
$$
\xipdown_{k+1}(x_{k+1})+\qedist(x_{k+1},y_{k+1})+\xipup_{k+1}(y_{k+1})= d.
$$
\end{proof}

By Lemmas \ref{lem:xi0} and \ref{lem:ind lemma} and induction,
we see that $\xi_t$ is defined and satisfies
\begin{enumerate}
\item
\mylabel{item:geet}
$\xi_t(z)\geq \epsilon$ for any $z\in P$,
\item
\mylabel{item:eqet}
$\xi_{t}(z)=\epsilon$ if $z\not\in\{x_1, \ldots, x_{t}, y_0,\ldots, y_{t-1}\}$.
\item
\mylabel{item:xipt}
$\xip_t(C)\leq d-\epsilon$ for any maximal chain $C$ in $P$ and
\item
\mylabel{item:x0y0t}
$\xipdown_t(x_i)+\qedist(x_i,y_i)+\xipup_t(y_i)= d$
for $0\leq i\leq t$,
where we set $\xipdown_t(x_0)=\xipup_t(y_t)=0$.
\end{enumerate}

\begin{example}
\rm
Let $\epsilon=1$ and let $P$ and $y_0$, $x_1$, $y_1$, $x_2$ be as follows.
$$
\begin{picture}(60,40)

\put(10,10){\circle*{2}}
\put(10,30){\circle*{2}}
\put(30,10){\circle*{2}}
\put(30,30){\circle*{2}}
\put(50,10){\circle*{2}}
\put(50,30){\circle*{2}}

\put(10,20){\circle*{2}}
\put(30,20){\circle*{2}}
\put(50,20){\circle*{2}}

\put(10,10){\line(0,1){20}}
\put(10,30){\line(1,-1){20}}
\put(30,10){\line(0,1){20}}
\put(30,30){\line(1,-1){20}}
\put(50,10){\line(0,1){20}}

\put(9,31){\makebox(0,0)[br]{$y_0$}}
\put(30,9){\makebox(0,0)[t]{$x_{1}$}}
\put(30,31){\makebox(0,0)[b]{$y_{1}$}}
\put(51,9){\makebox(0,0)[tl]{$x_{2}$}}

\end{picture}
$$
Then $ d=6$,
{\unitlength=.7\unitlength\relax
$$
\xi_0=
\vcenter{\hsize=60\unitlength
\begin{picture}(40,40)

\put(0,10){\circle*{3}}
\put(0,30){\circle*{3}}
\put(20,10){\circle*{3}}
\put(20,30){\circle*{3}}
\put(40,10){\circle*{3}}
\put(40,30){\circle*{3}}

\put(0,20){\circle*{3}}
\put(20,20){\circle*{3}}
\put(40,20){\circle*{3}}

\put(0,10){\line(0,1){20}}
\put(0,30){\line(1,-1){20}}
\put(20,10){\line(0,1){20}}
\put(20,30){\line(1,-1){20}}
\put(40,10){\line(0,1){20}}

\put(-1,31){\makebox(0,0)[br]{$3$}}
\put(-1,20){\makebox(0,0)[r]{$1$}}
\put(-1,9){\makebox(0,0)[tr]{$1$}}
\put(21,9){\makebox(0,0)[tl]{$1$}}
\put(21,20){\makebox(0,0)[l]{$1$}}
\put(19,31){\makebox(0,0)[br]{$2$}}
\put(41,9){\makebox(0,0)[tl]{$1$}}
\put(41,20){\makebox(0,0)[l]{$1$}}
\put(41,31){\makebox(0,0)[bl]{$1$}}

\end{picture}
},
\quad
\xi_1=
\vcenter{\hsize=60\unitlength
\begin{picture}(40,40)

\put(0,10){\circle*{3}}
\put(0,30){\circle*{3}}
\put(20,10){\circle*{3}}
\put(20,30){\circle*{3}}
\put(40,10){\circle*{3}}
\put(40,30){\circle*{3}}

\put(0,20){\circle*{3}}
\put(20,20){\circle*{3}}
\put(40,20){\circle*{3}}

\put(0,10){\line(0,1){20}}
\put(0,30){\line(1,-1){20}}
\put(20,10){\line(0,1){20}}
\put(20,30){\line(1,-1){20}}
\put(40,10){\line(0,1){20}}

\put(-1,31){\makebox(0,0)[br]{$3$}}
\put(-1,20){\makebox(0,0)[r]{$1$}}
\put(-1,9){\makebox(0,0)[tr]{$1$}}
\put(21,9){\makebox(0,0)[tl]{$2$}}
\put(21,20){\makebox(0,0)[l]{$1$}}
\put(19,31){\makebox(0,0)[br]{$2$}}
\put(41,9){\makebox(0,0)[tl]{$1$}}
\put(41,20){\makebox(0,0)[l]{$1$}}
\put(41,31){\makebox(0,0)[bl]{$1$}}

\end{picture}
}
\quad
\mbox{and}
\quad
\xi_2=
\vcenter{\hsize=60\unitlength
\begin{picture}(40,40)

\put(0,10){\circle*{3}}
\put(0,30){\circle*{3}}
\put(20,10){\circle*{3}}
\put(20,30){\circle*{3}}
\put(40,10){\circle*{3}}
\put(40,30){\circle*{3}}

\put(0,20){\circle*{3}}
\put(20,20){\circle*{3}}
\put(40,20){\circle*{3}}

\put(0,10){\line(0,1){20}}
\put(0,30){\line(1,-1){20}}
\put(20,10){\line(0,1){20}}
\put(20,30){\line(1,-1){20}}
\put(40,10){\line(0,1){20}}

\put(-1,31){\makebox(0,0)[br]{$3$}}
\put(-1,20){\makebox(0,0)[r]{$1$}}
\put(-1,9){\makebox(0,0)[tr]{$1$}}
\put(21,9){\makebox(0,0)[tl]{$2$}}
\put(21,20){\makebox(0,0)[l]{$1$}}
\put(19,31){\makebox(0,0)[br]{$2$}}
\put(41,9){\makebox(0,0)[tl]{$3$}}
\put(41,20){\makebox(0,0)[l]{$1$}}
\put(41,31){\makebox(0,0)[bl]{$1$}}

\end{picture}
}
$$
}

Let $\epsilon=1$ and let $P$ and $y_0$, $x_1$, $y_1$, $x_2$ be as follows.
$$
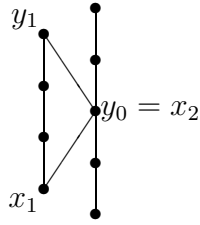
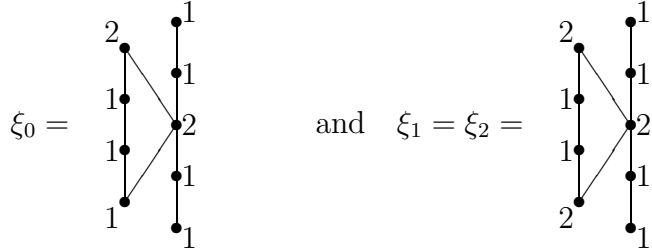
\begin{picture}(40,60)

\put(10,15){\circle*{2}}
\put(10,25){\circle*{2}}
\put(10,35){\circle*{2}}
\put(10,45){\circle*{2}}

\put(20,10){\circle*{2}}
\put(20,20){\circle*{2}}
\put(20,30){\circle*{2}}
\put(20,40){\circle*{2}}
\put(20,50){\circle*{2}}

\put(9,14){\makebox(0,0)[tr]{$x_1$}}
\put(9,46){\makebox(0,0)[br]{$y_1$}}
\put(21,30){\makebox(0,0)[l]{$y_0=x_2$}}

\put(10,15){\line(0,1){30}}
\put(20,10){\line(0,1){40}}
\put(10,15){\line(2,3){10}}
\put(10,45){\line(2,-3){10}}

\end{picture}
$$
Then $ d=7$,
$$
\xi_0=
\vcenter{\hsize=40\unitlength\relax
\begin{picture}(40,60)(10,0)

\put(10,15){\circle*{2}}
\put(10,25){\circle*{2}}
\put(10,35){\circle*{2}}
\put(10,45){\circle*{2}}

\put(20,10){\circle*{2}}
\put(20,20){\circle*{2}}
\put(20,30){\circle*{2}}
\put(20,40){\circle*{2}}
\put(20,50){\circle*{2}}

\put(9,14){\makebox(0,0)[tr]{$1$}}
\put(9,25){\makebox(0,0)[r]{$1$}}
\put(9,35){\makebox(0,0)[r]{$1$}}
\put(9,46){\makebox(0,0)[br]{$2$}}
\put(21,30){\makebox(0,0)[l]{$2$}}
\put(21,10){\makebox(0,0)[tl]{$1$}}
\put(21,20){\makebox(0,0)[l]{$1$}}
\put(21,40){\makebox(0,0)[l]{$1$}}
\put(21,50){\makebox(0,0)[bl]{$1$}}

\put(10,15){\line(0,1){30}}
\put(20,10){\line(0,1){40}}
\put(10,15){\line(2,3){10}}
\put(10,45){\line(2,-3){10}}

\end{picture}
}
\quad\mbox{and}
\quad
\xi_1=\xi_2=
\vcenter{\hsize=40\unitlength\relax
\begin{picture}(40,60)(10,0)

\put(10,15){\circle*{2}}
\put(10,25){\circle*{2}}
\put(10,35){\circle*{2}}
\put(10,45){\circle*{2}}

\put(20,10){\circle*{2}}
\put(20,20){\circle*{2}}
\put(20,30){\circle*{2}}
\put(20,40){\circle*{2}}
\put(20,50){\circle*{2}}

\put(9,14){\makebox(0,0)[tr]{$2$}}
\put(9,25){\makebox(0,0)[r]{$1$}}
\put(9,35){\makebox(0,0)[r]{$1$}}
\put(9,46){\makebox(0,0)[br]{$2$}}
\put(21,30){\makebox(0,0)[l]{$2$}}
\put(21,10){\makebox(0,0)[tl]{$1$}}
\put(21,20){\makebox(0,0)[l]{$1$}}
\put(21,40){\makebox(0,0)[l]{$1$}}
\put(21,50){\makebox(0,0)[bl]{$1$}}

\put(10,15){\line(0,1){30}}
\put(20,10){\line(0,1){40}}
\put(10,15){\line(2,3){10}}
\put(10,45){\line(2,-3){10}}

\end{picture}
}
$$

\end{example}

Let $i$ be an integer with $1\leq i\leq t-1$.
Take a saturated chain $x_i=z_0\covered z_1\covered\cdots\covered z_u=y_i$
with $u\epsilon=\qedist(x_i,y_i)$.
Then $z_1$, \ldots, $z_{u-1}\not\in\{x_1,\ldots, x_t,y_0, \ldots, y_{t-1}\}$.
In fact, if there is $z_\ell\in\{x_1,\ldots, x_t, y_0,\ldots, y_{t-1}\}$
with $1\leq \ell\leq u-1$, then 
$z_\ell=x_{i+1}$ or $z_\ell=y_{i-1}$ by \condn' and Lemma \ref{lem:cond np basic}.
If $z_\ell=x_{i+1}$, then 
$\qedist(x_i,x_{i+1})=\ell\epsilon$
and $\qedist(x_{i+1},y_i)=(u-\ell)\epsilon$ and therefore
\begin{eqnarray*}
\qedist(x_i,y_{i+1})
&\geq&\qedist(x_i,x_{i+1})+\qedist(x_{i+1},y_{i+1})\\
&=&\ell\epsilon+\qedist(x_{i+1},y_{i+1})\\
&=&u\epsilon-(u-\ell)\epsilon+\qedist(x_{i+1},y_{i+1})\\
&=&\qedist(x_i,y_i)-\qedist(x_{i+1},y_i)+\qedist(x_{i+1},y_{i+1}).
\end{eqnarray*}
This contradicts to the assumption that $y_0$, $x_1$, \ldots, $y_{t-1}$, $x_t$ is
\qered.

Therefore, $z_\ell\neq x_{i+1}$.
We see that $z_\ell\neq y_{i-1}$ by the same way.
Thus, $z_\ell\not\in\{x_1, \ldots, x_t, y_0, \ldots, y_{t-1}\}$
and therefore $\xi_t(z_\ell)=\epsilon$ for $1\leq \ell\leq u-1$
by \ref{item:eqet}.

Take saturated chains
$-\infty\covered z'_1\covered\cdots\covered z'_{u'}=x_i$ and
$y_i= z''_0\covered\cdots\covered z''_{u'}\covered\infty$
with
$
\xipdown_t(x_i)=\sum_{\ell=1}^{u'}\xi_t(z'_\ell)
$  and $
\xipup_t(y_i)=\sum_{\ell=0}^{u''}\xi_t(z''_\ell)
$.
Then $C_i\define\{z'_1,\ldots, z'_{u'}, z_1, \ldots, z_{u-1}, z''_0, \ldots, z''_{u''}\}$
is a maximal chain in $P$ with $x_i$, $y_i\in C_i$ and 
$\xi_t(z)=\epsilon$ for any $z\in C_i\cap(x_i,y_i)$.
Further,
\begin{eqnarray*}
\xip_t(C_i)
&=&\sum_{\ell=1}^{u'}\xi_t(z'_\ell)
+\sum_{\ell=1}^{u-1}\xi_t(z_\ell)
+\sum_{\ell=0}^{u''}\xi_t(z''_\ell)\\
&=&
\xipdown(x_i)+\qedist(x_i,y_i)-\epsilon+\xipup(y_i)\\
&=&
 d-\epsilon
\end{eqnarray*}
by \ref{item:x0y0} above.

By a similar way, we can take maximal chains $C_0$ and $C_t$ with
$\xip_t(C_0)=\xip_t(C_t)= d-\epsilon$
and for $\xi_t(z)=\epsilon$ for any $z\in (C_0\cap(-\infty,y_0))\cup(C_t\cap(y_t,\infty))$.

Define $\xi\colon P^-\to\ZZZ$ by
$$
\xi(z)\define
\begin{cases}
\xi_t(z),&\mbox{if $z\in P$;}\\
 d,&\mbox{if $z=-\infty$.}
\end{cases}
$$
Then by the facts above and Proposition \ref{prop:gen equiv},
we see the following.

\begin{prop}
\mylabel{prop:gen const}
Let $y_0$, $x_1$, \ldots, $y_{t-1}$, $x_t$
be a \qered\ \scn' in $P$.
Set $x_0\define-\infty$ and $y_t\define\infty$.
Then there exists $\xi\in \se$ such that
$T^\xi$ is a generator of $\omegae$  and
$$
\deg T^\xi=\qe(x_0,y_0,\ldots, x_t,y_t).
$$
\end{prop}

Now we state characterizations of level and anticanonical level
properties of $\kcp$.

\begin{thm}
\mylabel{thm:level cri}
The following conditions are equivalent.
\begin{enumerate}
\item
\mylabel{item:level}
$\kcp$ is level (resp.\ anticanonical level).
\item
\mylabel{item:no qonered}
\qonered\ (resp.\ \qmonered) \scn' is the empty sequence only.
\item
\mylabel{item:qone small}
For any sequence $y_0$, $x_1$, \ldots, $y_{t-1}$, $x_t$
with \condn',
$$
\qone(x_0,y_0,\ldots, x_t,y_t)\leq\qonedist(x_0,y_t)
$$
(resp.\
$\qmone(x_0,y_0,\ldots, x_t,y_t)\leq\qmonedist(x_0,y_t)$),
where we set $x_0\define-\infty$ and $y_t\define \infty$.
\end{enumerate}
\end{thm}
\begin{proof}
Let $\epsilon=1$ or $-1$.
We prove the level case and the anticanonical level case simultaneously
by setting $\epsilon=1$ when considering the level case and $\epsilon=-1$
when considering the anticanonical level case.
The contraposition of \ref{item:level}$\Rightarrow$\ref{item:no qonered}
follows from the fact that
$\qe(-\infty,y_0,x_1,\cdots,y_{t-1}, x_t,\infty)
>\qedist(-\infty,\infty)$ for any nonempty \qered\
sequence $y_0$, $x_1$, \ldots, $y_{t-1}$, $x_t$ with \condn',
since by
Proposition \ref{prop:gen const} and 
Corollary \ref{cor:gen deg min},
there are generators of $\omegae$ with degrees 
$\qe(-\infty,y_0,x_1,\cdots,y_{t-1}, x_t,\infty)$ and
$\qedist(-\infty,\infty)$ respectively.

We next show \ref{item:no qonered}$\Rightarrow$\ref{item:level}.
Let $T^\xi$, $\xi\in\se$ be an arbitrary generator of $\omegae$ and set 
$d\define\deg T^\xi$.
Then by Proposition \ref{prop:gen equiv}, we see that there exists a \qered\ sequence
$y_0$, $x_1$, \ldots, $y_{t-1}$, $x_t$ with \condn' and 
$C_0$, $C_1$, \ldots, $C_t\in C_\xi^{[d-\epsilon]}$
satisfying conditions of \ref{item:qered} of Proposition \ref{prop:gen equiv}.
By assumption $t=0$. Therefore $\xi(z)=\epsilon$ for any $z\in C_0$.
Thus, $\epsilon(\#C_0)=\xip(C_0)=d-\epsilon$.
On the other hand, since
$C_0\cup\{-\infty,\infty\}$ is a saturated chain from $-\infty$ to $\infty$,
we see that $d=\epsilon(\#C_0)+\epsilon=
\epsilon(\#C_0+1)\leq\qedist(-\infty,\infty)$.
Therefore, $d=\qedist(-\infty,\infty)$ by Corollary \ref{cor:gen deg min}.

\ref{item:qone small}$\Rightarrow$\ref{item:no qonered} is trivial.
In order to prove \ref{item:no qonered}$\Rightarrow$\ref{item:qone small},
let $y_0$, $x_1$, \ldots, $y_{t-1}$, $x_t$ be an arbitrary \scn'.
If there are $i$, $j$ with $i<j$ and $x_i<y_j$
and
$$
\qedist(x_i,y_j)\geq\qe(x_i,y_i,\ldots, x_j,y_j)
$$
then by setting $x'_u\define x_u$ for $1\leq u\leq i$,
$x'_u\define x_{u-j+i}$ for $i<u\leq t-j+i$,
$y'_u\define y_u$ for $0\leq u<i$ and
$y'_u\define y_{u-j+i}$ for $i\leq u<t-j+i$,
we get a sequence
$y'_0$, $x'_1$, \ldots $y'_{t-j+i-1}$, $x'_{t-j+i}$
with \condn' 
and
$$
\qe(x_0,y_0,\ldots, x_t,y_t)\leq
\qe(x'_0,y'_0,\ldots, x'_{t-j+i}, y'_{t-j+i}),
$$
where we set $x_0=x'_0=-\infty$ and $y_t=y'_{t-j+i}=\infty$.

Repeating this argument, we see that there is a \qered\ sequence
$y''_0$, $x''_1$, \ldots, $y''_{t''-1}$, $x''_{t''}$ with \condn'
such that
$$
\qe(x_0,y_0, \ldots, x_t,y_t)\leq
\qe(x''_0,y''_0, \ldots, x''_{t''},y''_{t''}),
$$
where we set $x''_0\define-\infty$ and $y''_{t''}\define\infty$.
Since \qered\ \scn' is the empty sequence only by assumption,
we see that $t''=0$, i.e.,
$$
\qe(x_0,y_0, \ldots, x_t,y_t)\leq\qedist(-\infty,\infty).
$$
\end{proof}

Note that if one replaces \condn' to \condn\ in \ref{item:no qonered}
and \ref{item:qone small} of Theorem \ref{thm:level cri}, one obtains a criterion
of level (resp.\ anticanonical level) property of $\kop$ \cite{mo,mf}.
Since \condn' is a weaker condition than \condn, we see the following.

\begin{cor}
\mylabel{cor:order implies chain}
If $\kcp$ is level (resp.\ anticanonical level), then so is $\kop$.
\end{cor}
The converse of this corollary does not hold.

\begin{example}
\mylabel{ex:level}
\rm
Let $n$, $m_1$, $m_2$ be integers with $n\geq 4$, $m_1$, $m_2\geq n-2$
and let
$P=\{a_1$, \ldots, $a_n$, $b_1$, \ldots, $b_{m_1}$, $c_1$, \ldots, $c_{m_2}$, $d\}$
be a poset with covering relations
$a_1\covered \cdots \covered a_n$,
$b_1\covered\cdots\covered b_{m_1}\covered d\covered c_1\covered \cdots \covered c_{m_2}$
and
$a_1\covered d\covered a_n$.
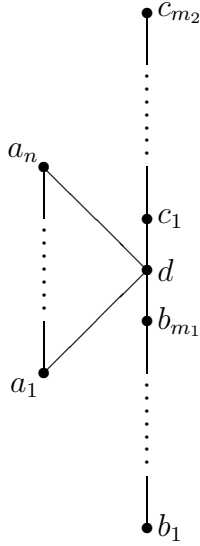
$$
\begin{picture}(40,120)
\put(10,40){\circle*{2}}
\put(9,39){\makebox(0,0)[tr]{$a_1$}}
\put(10,80){\circle*{2}}
\put(9,81){\makebox(0,0)[br]{$a_n$}}

\put(30,10){\circle*{2}}
\put(30,50){\circle*{2}}
\put(30,60){\circle*{2}}
\put(30,70){\circle*{2}}
\put(30,110){\circle*{2}}

\put(32,10){\makebox(0,0)[l]{$b_1$}}
\put(32,50){\makebox(0,0)[l]{$b_{m_1}$}}
\put(32,60){\makebox(0,0)[l]{$d$}}
\put(32,70){\makebox(0,0)[l]{$c_1$}}
\put(32,110){\makebox(0,0)[l]{$c_{m_2}$}}

\put(10,40){\line(0,1){10}}
\put(10,70){\line(0,1){10}}
\put(30,10){\line(0,1){10}}
\put(30,40){\line(0,1){40}}
\put(30,100){\line(0,1){10}}

\put(10,40){\line(1,1){20}}
\put(10,80){\line(1,-1){20}}

\multiput(10,52.5)(0,2.5){7}{\makebox(0,0){$\cdot$}}
\multiput(30,22.5)(0,2.5){7}{\makebox(0,0){$\cdot$}}
\multiput(30,82.5)(0,2.5){7}{\makebox(0,0){$\cdot$}}

\end{picture}
$$
Then $\kop$ is level and $\kcp$ is not.
In fact, it is easily verified that there is no \qonered\ \scn\ except the empty sequence.
Thus, $\kop$ is level.
However, $d$, $a_1$, $a_n$, $d$ is a \qonered\ \scn'.
Further, if we define $\xi\colon P^-\to\ZZZ$ by
$$
\xi(z)\define
\begin{cases}
m_1,&\mbox{if $z=a_1$;}\\
m_2,&\mbox{if $z=a_n$;}\\
n-2,&\mbox{if $z=d$;}\\
m_1+m_2+n-1,&\mbox{if $z=-\infty$;}\\
1,&\mbox{otherwise,}
\end{cases}
$$
then $T^\xi$ is a generator of $\omega$ with 
$\deg T^\xi=m_1+m_2+n-1$,
while the minimal degree of the generators of $\omega$ is 
$\qonedist(-\infty,\infty)=m_1+m_2+2<m_1+m_2+n-1$,
since $n\geq 4$.
\end{example}

\begin{example}
\rm
\mylabel{ex:antican level}
Let $n$ be a positive integer
and let
$P=\{a_1$, \ldots, $a_{n+3}$, $b_1$, \ldots, $b_n$, $c_1$, \ldots, $c_n$, $d_1$, $d_2$, $d_3\}$
be a poset with covering relations
$a_1\covered\cdots\covered a_{n+3}$, $d_1\covered d_2\covered d_3$,
$a_1\covered b_1\covered\cdots\covered b_n\covered d_2\covered 
c_1\covered \cdots\covered c_n\covered a_{n+3}$.
$$
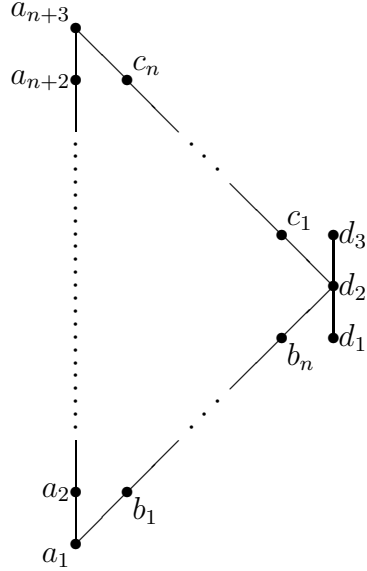
\begin{picture}(70,120)
\put(10,10){\circle*{2}}
\put(10,20){\circle*{2}}
\put(10,100){\circle*{2}}
\put(10,110){\circle*{2}}
\put(9,9){\makebox(0,0)[tr]{$a_1$}}
\put(9,20){\makebox(0,0)[r]{$a_2$}}
\put(9,100){\makebox(0,0)[r]{$a_{n+2}$}}
\put(9,111){\makebox(0,0)[br]{$a_{n+3}$}}

\put(60,50){\circle*{2}}
\put(60,60){\circle*{2}}
\put(60,70){\circle*{2}}

\put(61,50){\makebox(0,0)[l]{$d_1$}}
\put(61,60){\makebox(0,0)[l]{$d_2$}}
\put(61,70){\makebox(0,0)[l]{$d_3$}}

\put(20,20){\circle*{2}}
\put(50,50){\circle*{2}}
\put(50,70){\circle*{2}}
\put(20,100){\circle*{2}}

\put(21,19){\makebox(0,0)[tl]{$b_1$}}
\put(51,49){\makebox(0,0)[tl]{$b_n$}}
\put(51,71){\makebox(0,0)[bl]{$c_1$}}
\put(21,101){\makebox(0,0)[bl]{$c_n$}}

\put(10,10){\line(0,1){20}}
\put(10,90){\line(0,1){20}}
\put(60,50){\line(0,1){20}}

\put(10,10){\line(1,1){20}}
\put(40,40){\line(1,1){20}}
\put(40,80){\line(1,-1){20}}
\put(10,110){\line(1,-1){20}}

\multiput(10,32.5)(0,2.5){23}{\makebox(0,0){$\cdot$}}
\multiput(32.5,32.5)(2.5,2.5){3}{\makebox(0,0){$\cdot$}}
\multiput(32.5,87.5)(2.5,-2.5){3}{\makebox(0,0){$\cdot$}}
\end{picture}
$$
Then $\kop$ is anticanonical level and $\kcp$ is not.
In fact, it is easily verified that there is no \qmonered\ \scn\
except the empty sequence.
Thus, $\kop$ is anticanonical level.
However, $d_2$, $a_1$, $a_{n+3}$, $d_2$ is a \qmonered\ \scn'.
Further, if we set
$\xi\colon P^-\to\ZZZ$ by
$$
\xi(z)\define
\begin{cases}
n-1,&\mbox{if $z=a_1$, $a_{n+3}$ or $d_2$}\\
n-4,&\mbox{if $z=-\infty$}\\
-1,&\mbox{otherwise,}
\end{cases}
$$
then $T^\xi$ is a generator of $\omega^{(-1)}$ with degree $n-4$,
while the minimal degree of the generators of $\omega^{(-1)}$ is 
$\qmonedist(-\infty,\infty)=-4$.
\end{example}


\section{Degrees of the generators of canonical and anticanonical ideals}

In this section, we show that the degrees of generators of $\omega$ and $\omega^{(-1)}$
are consecutive integers.
Since the cases of $\omega$ and $\omega^{(-1)}$ are similar, we set $\epsilon=1$ or $-1$
and treat the cases of $\omega$ and $\omega^{(-1)}$ simultaneously.
Recall that the minimal degree of the generators of $\omegae$ is $\qedist(-\infty,\infty)$
by Corollary \ref{cor:gen deg min}.

We first note the following fact.

\begin{lemma}
\mylabel{lem:cross chain}
Let $\xi\in\se$ and $d=\xi(-\infty)$.
Suppose that $C_1$, $C_2\in C_\xi^{[d-\epsilon]}$ and
$z\in C_1\cap C_2$.
Then
$(C_1\cap(-\infty,z])\cup(C_2\cap(z,\infty))\in C_\xi^{[d-\epsilon]}$.
\end{lemma}
\begin{proof}
Set 
$C=(C_1\cap(-\infty,z])\cup(C_2\cap(z,\infty))$ and
$C'=(C_2\cap(-\infty,z])\cup(C_1\cap(z,\infty))$.
Then, it is easily verified that $C$ and $C'$ are maximal chains in $P$.
Further,
$$
\xip(C)=
\sum_{c\in C_1\cap(-\infty,z]}\xi(w)+
\sum_{c\in C_2\cap(z,\infty)}\xi(w)
$$
and
$$
\xip(C')=
\sum_{c\in C_2\cap(-\infty,z]}\xi(w)+
\sum_{c\in C_1\cap(z,\infty)}\xi(w).
$$
Since
$
\sum_{c\in C_1\cap(-\infty,z]}\xi(w)+
\sum_{c\in C_1\cap(z,\infty)}\xi(w)
=\xip(C_1)=d-\epsilon
$
and
$
\sum_{c\in C_2\cap(-\infty,z]}\xi(w)+
\sum_{c\in C_2\cap(z,\infty)}\xi(w)
=\xip(C_2)=d-\epsilon
$,
we see that
$$
\xip(C)+\xip(C')=2(d-\epsilon).
$$
On the other hand, since $\xi\in\se$, we see that $\xip(C'')\leq d-\epsilon$
for any maximal chain $C''$ in $P$.
Thus, we see that
$
\xip(C)=\xi(C')=d-\epsilon$.
\end{proof}
Next we show the following fact.

\begin{lemma}
\mylabel{lem:max contain nonmin}
Let $\xi$ be an element of $\se$ with $d=\xi(-\infty)>\qedist(-\infty,\infty)$.
Then for any $C\in C_\xi^{[d-\epsilon]}$, there exists $z\in C$ such that $\xi(z)>\epsilon$.
\end{lemma}
\begin{proof}
Assume the contrary and take $C\in C_\xi^{[d-\epsilon]}$ with $\xi(c)\leq \epsilon$ for any
$c\in C$.
Since $\xi\in\se$, we see that $\xi(c)=\epsilon$ for any $c\in C$.
Therefore, 
$d-\epsilon=\xip(C)=\epsilon(\#C)=\epsilon(\#(C\cup\{-\infty,\infty\})-1)-\epsilon
\leq \qedist(-\infty,\infty)-\epsilon$.
This contradicts to the assumption that $d>\qedist(-\infty,\infty)$.
\end{proof}

Now we prove the following.

\begin{lemma}
\mylabel{lem:dec1}
Let $\xi$ be an element of $\se$ such that $T^\xi$ is a generator of $\omegae$ with
degree $d>\qedist(-\infty,\infty)$.
Define $\xi_1\colon P^-\to\ZZZ$ by
$$
\xi_1(z)\define
\begin{cases}
\xi(z)-1,&\vtop{\hsize.5\textwidth\relax\noindent
if there exists $C\in C_\xi^{[d-\epsilon]}$ with 
$z=\max\{c\in C\mid\xi(c)>\epsilon\}$ or $z=-\infty$;}\\
\xi(z),&\mbox{otherwise.}
\end{cases}
$$
Then $\xi_1\in\se$ and $T^{\xi_1}$ is a generator of $\omegae$ with degree $d-1$.
\end{lemma}
\begin{proof}
First we show that $\xi_1\in\se$.
It is clear that $\xi_1(z)\geq\epsilon$ for any $z\in P$ by the definition of $\xi_1$.
Let $C$ be an arbitrary maximal chain in $P$.
Then $\xip(C)\leq d-\epsilon$ since $\xi\in\se$.
If $\xip(C)\leq d-\epsilon-1$, then $\xip_1(C)\leq d-\epsilon-1$, since
$\xi_1(z)\leq \xi(z)$ for any $z\in P$.
Suppose that $C\in C_\xi^{[d-\epsilon]}$.
Then by Lemma \ref{lem:max contain nonmin}, we see that there is $z\in C$
with $\xi(z)>\epsilon$.
Therefore, if we set $z=\max\{c\in C\mid\xi(c)>\epsilon\}$, then
$\xi_1(z)=\xi(z)-1$.
Thus, we see that
$$
\xip_1(C)\leq\xip(C)-1\leq d-\epsilon-1=\xi_1(-\infty)-\epsilon.
$$

Now we show that $T^{\xi_1}$ is a generator of $\omegae$.
Since $T^\xi$ is a generator of $\omegae$, we see by Proposition \ref{prop:gen equiv}
that there is a \qered\ sequence $y_0$, $x_1$, \ldots, $y_{t-1}$, $x_t$ with \condn'
and $C_0$, $C_1$, \ldots, $C_t\in C_\xi^{[d-\epsilon]}$ 
which satisfy the conditions of \ref{item:qered} of Proposition \ref{prop:gen equiv}.

If $t=0$, then $C_0\in C_\xi^{[d-\epsilon]}$ and $\xi(z)=\epsilon$ for any $z\in C_0$.
This contradicts to Lemma \ref{lem:max contain nonmin}.
Thus, we see that $t>0$.
Further, we see that
$\{c\in C_i\mid \xi(c)>\epsilon\}\neq\emptyset$ for $0\leq i\leq t$.
In particular, $\{c\in C_i\mid\xi_1(c)=\xi(c)-1\}\neq\emptyset$ for $0\leq i\leq t$.

Set $z_i\define\min\{c\in C_i\mid\xi_1(c)=\xi(c)-1\}$ for $0\leq i\leq t$.
Since $\xi(c)=\epsilon$ for any $c\in C_t\cap(x_t,\infty)$, 
we see that $z_t\leq x_t$.
Set $j\define\min\{i\mid z_i\leq x_i\}$.
Then $j>0$.
Further, for any $i$ with $0\leq i<j$, $z_i\geq y_i$.
For $i$ with $0\leq i\leq j$, take $C'_i\in C_\xi^{[d-\epsilon]}$ with 
$z_i=\max\{c\in C'_i\mid \xi(c)>\epsilon\}$.
Then by Lemma \ref{lem:cross chain}, we see that
$$
C''_i\define(C_i\cap(-\infty,z_i])\cup(C'_i\cap(z_i,\infty))\in C_\xi^{[d-\epsilon]}
$$
and moreover, by the definition of $z_i$, and the choice of $C'_i$,
$
\{c\in C''_i\mid\xi_1(c)\neq\xi(c)\}=\{z_i\}.
$
In particular, $\xip_1(C''_i)=d-\epsilon-1$ for $0\leq i\leq j$ since $\xi_1(z_i)=\xi(z_i)-1$.

Since
$y_0>x_1<y_1>x_2<\cdots<y_{j-1}>z_j$,
$y_0\in C''_0$, $z_j\in C''_j$, $x_i$, $y_i\in C''_i$ for $1\leq i\leq j-1$
and 
$C''_i\in C_{\xi_1}^{[d-1-\epsilon]}$ for $0\leq i\leq j$
and
$\xi_1(z)=\epsilon$ for any $z\in C''_i\cap(x_i,y_i)$ for $i$ with $0\leq i\leq j-1$ 
and $z\in C''_j\cap(z_j,\infty)$,
we see by Proposition \ref{prop:gen equiv} that $T^{\xi_1}$ is a generator of $\omegae$.
\end{proof}
Now we show the following.

\begin{thm}
\mylabel{thm:gen deg}
Let $P$ be a finite poset.
Set $d_0\define\qonedist(-\infty,\infty)$
(resp.\ $d_0\define\qmonedist(-\infty,\infty)$) 
and 
$\dmax\define\max\{\sum_{\ell=0}^t\qone(x_0,y_0,\ldots,x_t,y_t)\mid y_0$, $x_1$, \ldots,
$y_{t-1}$, $x_t$ is a \qonered\ \scn', where $x_0=-\infty$ and $y_t=\infty\}$
(resp.\ 
$\dmax\define\max\{\sum_{\ell=0}^t\qmone(x_0,y_0,\ldots,x_t,y_t)\mid y_0$, $x_1$, \ldots,
$y_{t-1}$, $x_t$ is a \qmonered\ \scn', where $x_0=-\infty$ and $y_t=\infty\}$).
Then for any integer $d$ with $d_0\leq d\leq \dmax$, there exists a generator of $\omega$
(resp.\ $\omega^{(-1)}$) of degree $d$.
\end{thm}
\begin{proof}
We prove the assertions of $\omega$ and $\omega^{(-1)}$ simultaneously by setting 
$\epsilon=1$ or $-1$.
By Corollary \ref{cor:gen deg min}, we see that arbitrary generator of $\omegae$ has degree
greater than or equals to $d_0$.
Further, by Corollary \ref{cor:gen deg max}, we see that any generator of $\omegae$
has degree less than or equals to $\dmax$.
On the other hand, we see by Proposition \ref{prop:gen const} that there is a generator
of $\omegae$ with degree $\dmax$.
Thus, by Lemma \ref{lem:dec1} and the backward induction, we see the result.
\end{proof}

By the same argument as the end of \S\ref{sec:symbolic power}, we see the following fact.

\begin{cor}
Let $G$ be a finite graph.
Suppose that each odd cycle of $G$ has a triangular chord.
Then the degree of the generators of the canonical and anticanonical ideals of the Ehrhart ring of
the stable set polytope of $G$ are consecutive integers.
\end{cor}


%

\end{document}